\documentclass[10.5pt,reqno]{amsart}

\usepackage[utf8]{inputenc}
\newcommand{\RR}{\mathbb{R}}

\usepackage{color}
\usepackage{comment}
\usepackage{esint,amssymb}
\usepackage{graphicx}
\usepackage{MnSymbol}
\usepackage{mathtools}
\usepackage[colorlinks=true, pdfstartview=FitV, linkcolor=blue, citecolor=blue, urlcolor=blue,pagebackref=false]{hyperref}
\usepackage{microtype}

\usepackage{bm}
\usepackage{scalerel} 
\usepackage{dsfont}
\usepackage{mathrsfs}
\usepackage[font={footnotesize}]{caption}
\usepackage{comment}
\usepackage[shortlabels]{enumitem}
\usepackage{mdwlist}
\definecolor{darkgreen}{rgb}{0,0.5,0}
\definecolor{darkblue}{rgb}{0,0,0.7}
\definecolor{pink}{rgb}{0.8,0.25,0.5125}
\definecolor{purple}{rgb}{0.5,0.3,0.9}

\newcommand{\jl}[1]{{\color{pink}{#1}}}

\newtheorem{theorem}{Theorem}[section]
\newtheorem{proposition}[theorem]{Proposition}
\newtheorem{lemma}[theorem]{Lemma}

\theoremstyle{remark}
\newtheorem{remark}[theorem]{Remark}

\theoremstyle{definition}
\newtheorem{definition}[theorem]{Definition}

\numberwithin{equation}{section}
\numberwithin{equation}{section}

\newcommand{\NN}{\mathbb{N}}
\newcommand{\dstf}{h_{\nu}}

\newcommand{\ve}{\varepsilon}

\newcommand{\vp}{\varphi}

\newcommand{\al}{\alpha}

\newcommand{\ga}{\gamma}

\newcommand{\la}{\lambda}

\DeclareMathOperator{\sech}{sech}

\newcommand{\Ga}{\Gamma}
\newcommand{\De}{\Delta}

\newcommand{\La}{\Lambda}

\newcommand{\Om}{\Omega}
\newcommand{\norm}[1]{\|#1\|}
 
\newcommand{\rst}[1]{\ensuremath{{\mathbin\upharpoonright}%
\raise-.5ex\hbox{$#1$}}}

\newcommand{\sh}{\mathcal{H}}

\newcounter{gscan}
\newcounter{bwscan}
\newcounter{cscan}
\newcounter{hscan}
\newcounter{fscan}
\newcounter{pscan}
\newcounter{sscan}
\newcounter{iscan}
\newcounter{rscan}
\newcounter{rrscan}
\newcounter{fpscan}

\renewcommand{\thehscan}{H\arabic{hscan}}

\renewcommand{\tilde}{\widetilde}

\title{$\sigma$}
\title[]{Anisotropic Surface Tensions for Phase Transitions in Periodic Media}

\author[R. Choksi]{Rustum Choksi}
\address[R. Choksi]{Department of Mathematics and Statistics, McGill University, Burnside Hall, 805 Sherbrooke Street West, Montreal, QC H3A 0B9}
\email{rustum.choksi@mcgill.ca}

\author[I. Fonseca]{Irene Fonseca}
\address[I. Fonseca]{Department of Mathematical Sciences, Carnegie Mellon University, Wean Hall, 5000 Forbes Avenue, Pittsburgh, PA 15207}
\email{fonseca@andrew.cmu.edu}

\author[J. Lin]{Jessica Lin}
\address[J. Lin]{Department of Mathematics and Statistics, McGill University, Burnside Hall, 805 Sherbrooke Street West, Montreal, QC H3A 0B9}
\email{jessica.lin@mcgill.ca}

\author[R. Venkatraman]{Raghavendra Venkatraman}
\address[R. Venkatraman]{Courant Institute of Mathematical Sciences, 251 Mercer Street, New York, NY-10012}
\email{raghav@cims.nyu.edu}

\date{\today}

\begin{document}

\maketitle
\textbf{Abstract:} This paper establishes bounds on the homogenized surface tension for a heterogeneous Allen-Cahn energy functional in a periodic medium. The approach is based on relating the homogenized energy to a purely geometric variational problem involving the large scale behaviour of the signed distance function to a hyperplane in periodic media.  Motivated by this, a homogenization result for the signed distance function to a hyperplane in both periodic and almost periodic media is proven. 
%

\section{Introduction}

\subsection{The Setting and Statements of the Main Results} 
We examine anisotropic surface tensions arising from the periodic homogenization of energy functionals in the study of phase transitions. Here, we focus on a subclass of problems presented in \cite{CFHP} where the authors study inhomogeneous media characterized by a heterogeneous double-well potential. 
Precisely, we consider double-well potentials of the form 
\begin{equation}\label{e.well}
\tilde{W}(x,u)=a(x)W(u):=a(x)(1-u^{2})^{2}, 
\end{equation}
with $a: \RR^{N}\rightarrow \RR$ continuous, strictly positive, and $\mathbb{T}^{N}$-periodic, where $\mathbb{T}^N$ denotes the standard $N-$dimensional torus.  
In \cite{CFHP}, the authors addressed the $\Ga$-limit of the gradient regularized problem with energy 
 $\mathcal{F}_{\ve}: H^{1}(\Om, \RR^{N})\rightarrow [0, \infty]$, defined by 
\begin{equation}\label{e.fepdef}
\mathcal{F}_{\ve}(u):=\int_{\Omega}\left[\frac{1}{\ve}a\left(\frac{x}{\ve}\right)W\left(u\right)+\frac{\ve}{2}\left|\nabla u\right|^{2}\right]\, dx. 
\end{equation}
Their result pertained to a more general class of potentials $\tilde{W}(x,u)$, but the work presented here relies critically on the product structure (\ref{e.well}). 
The $\Ga$-limit obtained in \cite{CFHP} has the typical form of the weighted perimeter functional 
\begin{equation*}
\mathcal{F}_{0}(u):=\begin{cases}
\int_{\partial^{*}\left\{u=-1\right\}\cap \Omega}\sigma(\nu_{\left\{u=-1\right\}}(x))\, d\mathcal{H}^{N-1}(x) &\text{if $u \in BV(\Om; \left\{-1,1\right\})$,} \\
+\infty &\text{otherwise}, 
\end{cases}
\end{equation*}
where $\nu_{\left\{u=-1\right\}}$ denotes the measure-theoretic external unit normal to the reduced boundary of the level set $\left\{u=-1\right\}$,  and the anisotropic surface tension $\sigma: \mathbb{S}^{N-1}\to [0,\infty)$ is defined by a cell formula governed by a variational problem (see \eqref{e.sigma}). 

In homogenization, the first step consists of finding $\sigma,$ which characterizes the effective ``homogenized'' behavior of the system. A natural follow-up question is to obtain further refined information which clarifies the asymptotic cell formula: this \textit{a posteriori} investigation seeks to understand a number of issues, such as bounds on the homogenized coefficients, regularity, and ellipticity of the effective surface tension $\sigma.$ This paper is concerned with the first of these: bounds on the effective, anisotropic surface tension $\sigma$ obtained from the analysis in \cite{CFHP,CFHP2}, which we achieve via comparison with a novel ``geodesic'' formula. These bounds are written in terms of a metric which takes into account the heterogeneities of the underlying media. With regards to some of the other properties posed above (i.e. regularity), we mention the works of \cite{CGN, willpeter2} which contain very interesting results in these directions.

To state and then motivate our results, we need to first recall the  effective surface energy density $\sigma: \mathbb{S}^{N-1}\rightarrow (0, \infty)$ which was introduced in \cite{CFHP}. To this end, we develop some notation.  

Let $N \geqslant 2$ denote the spatial dimension, and let $\left\{e_{1}, e_{2}, \cdots, e_{N}\right\}$ be the standard orthonormal basis of $\RR^{N}$.
\begin{itemize}
\item \underline{Cubes:} With respect to this basis, let $Q:=(-\tfrac{1}{2}, \tfrac{1}{2})^{N}$ be the unit cube in $\RR^{N}$ centered at the origin, and for each $\nu\in \mathbb{S}^{N-1}$, let $Q_{\nu}$ be a unit cube centered at the origin with two faces orthogonal to $\nu$. Let $\Sigma_\nu$ denote the plane through the origin with normal $\nu,$ and we set $\square_\nu := \Sigma_\nu \cap Q_\nu,$ an $(N-1)-$dimensional unit cube in the plane $\Sigma_\nu.$
\item \underline{Half-Spaces:} For each $\nu\in \mathbb{S}^{N-1}$, we define $\mathcal{H}_{\nu}:=\left\{x\cdot \nu>0\right\}$. This is the ``positive'' open half-space in the direction $\nu$. 
\item \underline{Sequences:} In what follows, when we write $T \to \infty,$ we understand an arbitrary sequence $\{T_m\}_{m \in \mathbb{N}}, $ with $T_m \to \infty$ as $m \to \infty.$ 
\end{itemize}
The following hypotheses \eqref{h.per}-\eqref{h.smootha} are used in the sequel: 
\begin{list}{ (\thehscan)}
{
\usecounter{hscan}
\setlength{\topsep}{1.5ex plus 0.2ex minus 0.2ex}
\setlength{\labelwidth}{1.2cm}
\setlength{\leftmargin}{1.5cm}
\setlength{\labelsep}{0.3cm}
\setlength{\rightmargin}{0.5cm}
\setlength{\parsep}{0.5ex plus 0.2ex minus 0.1ex}
\setlength{\itemsep}{0ex plus 0.2ex}
}
\item \label{h.per} $a : \RR^N \to (0,\infty)$ is $\mathbb{T}^N$-periodic, i.e., $a(x+ke_{i})=a(x)$ for all $x\in \mathbb{R}^{N}$, $k\in \mathbb{Z}$, $i \in \{1,\cdots,N\}.$  
\item \label{h.nondeg} There exist {$\Theta>\theta >0$,} such that for all $x\in \RR^{N}$, {$\theta\leqslant a(x)\leqslant \Theta$. } 
\item \label{h.domain} $\Om\subseteq \RR^{N}$ is a Lipschitz domain. 
\item \label{h.smootha} $a$ is continuous. 
\end{list}
 Let
\begin{equation}\label{e.Cdef}
\mathcal{C}(TQ_{\nu}):=\left\{u\in H^{1}\left(TQ_{\nu}; \RR\right): u=\rho * u_{0,\nu}\, \text{on $\partial (TQ_{\nu})$}\right\},
\end{equation}
with
\begin{equation*}
u_{0, \nu}(y):=\begin{cases}
-1 &\text{if $x\cdot\nu\leqslant 0$},\\
1&\text{if $x\cdot \nu>0$},
\end{cases}
\end{equation*}
and $\rho\in C^{\infty}_{c}(B(0,1)),$ with $0 \leqslant \rho \leqslant 1,$  and $\int_{\RR^{N}} \rho(x)\, dx=1$. Following \cite{CFHP}, we define $\sigma: \mathbb{S}^{N-1}\rightarrow (0, \infty)$ by the cell formula 
\begin{equation}
\label{e.sigma}
\sigma(\nu) := \lim_{T \to \infty} \frac{1}{T^{N-1}} \inf \left\{\int_{TQ_\nu} \left[ a(y)W(u) +\frac{1}{2}|\nabla u|^2\right]  \,dy : u \in \mathcal{C}(TQ_{\nu}) \right\}. 
\end{equation}
We now state the precise result of \cite{CFHP}. 
\begin{theorem}(\cite[Theorem 1.6]{CFHP}, see also \cite{CFHP2}).\label{t.fonseca}
Let $\left\{\ve_{k}\right\}_{k\in \mathbb{N}}$ be a sequence such that $\ve_{k}\rightarrow 0^+$ as $k\rightarrow \infty$. Assume that \eqref{h.per}-\eqref{h.domain} hold, and that the function  $a$ is measurable. 
\begin{enumerate}
\item [(i)] If $\left\{u_{k}\right\}_{k\in \mathbb{N}}\subseteq H^{1}(\Om; \RR)$ satisfies 
\begin{equation*}
\sup_{k\in \mathbb{N}} \mathcal{F}_{\ve_{k}}(u_{k})<+\infty, 
\end{equation*}
then, up to a subsequence (not relabeled), there exists some function $u \in BV(\Om; \left\{-1,1\right\})$ so that $u_{k}\xrightarrow{L^1(\Om)} u $. 
\item [(ii)]As $k\rightarrow \infty$, $\mathcal{F}_{\ve_{k}}\xrightarrow{\Ga-L^{1}}\mathcal{F}_{0}$, where 
\begin{equation}\label{e.f0def}
\mathcal{F}_{0}(u)=\begin{cases}
\int_{\partial^{*}A}\sigma(\nu_{A}(x))\, d\mathcal{H}^{N-1}(x) &\text{if $u \in BV(\Om; \left\{-1,1\right\})$,} \\
+\infty &\text{otherwise},
\end{cases}
\end{equation}
for $\sigma:\mathbb{S}^{N-1}\rightarrow [0, \infty)$ defined by \eqref{e.sigma}, $A:=\left\{u=-1\right\}$, and $\nu_{A}(x)$ is the measure theoretic external unit normal to the reduced boundary $\partial^{*}A$ at $x$. 
\item [(iii)] $\sigma: \mathbb{S}^{N-1}\rightarrow [0, \infty)$ defined by \eqref{e.sigma} is continuous, and its one-homogeneous extension is convex, and hence, locally Lipschitz in $\RR^N$. 
\end{enumerate}
\end{theorem}

The formula \eqref{e.sigma} for $\sigma$ embeds a one-parameter family of variational problems, henceforth called the {\it cell problem}. Our first main result consists of \textit{anisotropic} bounds on $\sigma$ in relation to a novel geodesic formula which is expressed by solutions to an associated Eikonal equation. 
To formulate our estimates, consider the Riemannian metric in $\RR^N$ given by the following: for any $y_0, y_1 \in \RR^N,$ 
\begin{align} \label{e.deflength}
d_{\sqrt{a}}(y_0,y_1) := \inf_{\gamma} \int_0^1 \sqrt{a(\gamma(t))}|\dot{\gamma}(t)|\,dt,
\end{align}
where the infimum is taken among Lipschitz curves $\gamma:[0,1] \to \RR^N$ with $\gamma(j) = y_j, j = 0,1.$ Standard arguments via the Hopf-Rinow theorem entail that $\RR^N,$ with the $d_{\sqrt{a}}$ metric, is a complete metric space. For any $\nu \in \mathbb{S}^{N-1},$ recalling that  \[ \Sigma_\nu := \{x \in \RR^N: x \cdot \nu = 0\},\] we consider the signed distance function with respect to the $d_{\sqrt{a}}$ metric, to the plane $\Sigma_\nu.$ Precisely, 
\begin{align} \label{e.hnudefintro}
h_\nu(y) := \mathrm{sign}(y \cdot \nu) \inf_{z \in \Sigma_\nu}d_{\sqrt{a}}(y,z). 
\end{align}
It is well known, and recalled in Lemma \ref{l.hlips} below, that $h_\nu$ is Lipschitz continuous and satisfies, pointwise a.e., the eikonal equation  \[ |\nabla h_\nu| = \sqrt{a}\quad\text{in $\RR^{N}$.}\]
We next present the first main result.
\begin{theorem}\label{t.sigmarep}
Suppose that \eqref{h.per}-\eqref{h.smootha} hold, and let $\sigma: \mathbb{S}^{N-1}\rightarrow [0, \infty)$ be the anisotropic surface energy as in \eqref{e.sigma}. Let $q : \RR \to \RR$ be defined by 
$$
q(z) := \tanh(\sqrt{2}z), \quad z \in \RR.
$$
For $\nu \in \mathbb{S}^{N-1},$ define 
\begin{equation}
\label{e.lambda}
\begin{aligned} 
&\underline{\lambda}(\nu) := \liminf_{T \to \infty} \frac{1}{T^{N-1}} \int_{TQ_\nu}\left[a(y) W(q \circ h_\nu)+\frac{1}{2}|\nabla (q \circ h_\nu)|^2\right]  \,dy,\\
&\overline{\lambda}(\nu) := \limsup_{T \to \infty} \frac{1}{T^{N-1}} \int_{TQ_\nu}\left[a(y) W(q \circ h_\nu)+\frac{1}{2} |\nabla (q \circ h_\nu)|^2\right]  \,dy.
\end{aligned}\end{equation}
There exist $\Lambda_{0}>0$ universal and $\la_{0}: \mathbb{S}^{N-1}\rightarrow [0, \La_{0}]$ such that 
\begin{align} \label{e.sigrep}
\overline{\lambda}(\nu) - \lambda_0(\nu) \leqslant \sigma(\nu) \leqslant \underline{\lambda}(\nu). 
\end{align}
\end{theorem}
\begin{remark}
 {We conjecture that the main result of \cite{CFHP}, namely, Theorem \ref{t.fonseca}, holds with an identical cell formula, when $a$ is merely almost periodic, as opposed to periodic. If this is the case, then our bounds in Theorem \ref{t.sigmarep} also apply to the setting where $a$ is almost periodic. }
 \end{remark}

The computation of these bounds depends solely  on the large-scale behavior of the distance functions $h_\nu$,  for which one can readily invoke efficient numerical algorithms, for example fast marching and sweeping methods \cite{sethian}. 
As we explain in the next subsection, the structure of the new geodesic formula, which is the basis for our bounds, is an intuitive generalization of the Modica-Mortola framework for the  homogeneous case $ a \equiv 1$.

Motivated by these bounds, we next turn to rigorous analytical results concerning the large-scale behavior of the distance functions $h_\nu$. Precisely, we seek to characterize the limit 
\begin{align*}
\lim_{T \to \infty} \frac{h_\nu(Tx)}{T}, \quad \quad x \in \RR^N,
\end{align*} 
in a suitable topology of functions. Our second main result resolves this question, by showing that these rescaled functions converge locally uniformly to the signed distance function in an effectively homogeneous medium. 

\begin{theorem}\label{t.planarint}
Suppose \eqref{h.per}-\eqref{h.smootha} hold. For each $\nu\in \mathbb{S}^{N-1},$ there exists a unique $c(\nu)\in [\sqrt{\theta}, \sqrt{\Theta}]$ such that for all $K\subseteq \RR^{N}$ compact, we have 
\begin{align*}
\lim_{T \to \infty} \sup_{x\in K}\left|\frac{1}{T}h_\nu(T x)-c(\nu)(x\cdot \nu)\right|=0,
\end{align*}
and $c(\nu)=c(-\nu)$.

\end{theorem}

From the perspective of geometry, Theorem \ref{t.planarint} characterizes the large-scale limiting behaviour of the signed distance function to a hyperplane in a periodic Riemannian metric that is conformal ot the Euclidean one. We refer to the works of Bangert \cite{Bangert} and Burago \cite{burago}, who studied the behaviour of ``point-to-point'' distances in periodic metrics, in greater abstraction than what we study here. Under the same rescaling, they identify the effective ``stable norm'' $\norm{x-y}_{*}$ which characterizes the effective distance between $x,y\in \RR^{N}$. In Remark \ref{r.fg}, we discuss some open directions relating this effectively homogeneous distance function $c(\nu)(\cdot)\cdot\nu$ with the stable norm $\norm{\cdot}_{*}$ identified in these works. 

Theorem \ref{t.planarint} also implies a homogenization result for the Eikonal equation in half-spaces. Indeed, it is well known (see for example \cite{ManMen}) that for each fixed $\nu\in \mathbb{S}^{N-1}$, for $\left\{T_{m}\right\}_{m\geq 0}$ with $T_{m}\to \infty$ as $m\to \infty$, the functions $k_{m}(x):=T_{m}^{-1}h_{\nu}(T_{m}x)$ and $\ell(x):=c(\nu)(x\cdot \nu)$ are the unique viscosity solutions of 
\begin{alignat}{2}\label{e.planmet2}
    & \begin{aligned} & \begin{cases}
|\nabla k_{m}|=\sqrt{a(T_{m}x)}&\text{in $\mathcal{H}_{\nu}$},\\
k_{m}=0&\text{on $\Sigma_{\nu}$},
  \end{cases}\\
  \end{aligned}
    & \hskip 3em \text{and} \hskip 3em&
     \begin{aligned}
  & \begin{cases}
|\nabla \ell|=c(\nu)&\text{in $\mathcal{H}_{\nu}$},\\
\ell=0&\text{on $\Sigma_{\nu}$}, 
  \end{cases} 
  \end{aligned}\end{alignat}
  where we recall that $\mathcal{H}_{\nu}=\left\{x\cdot \nu>0\right\}$. We note that although $h_{\nu}$ is defined on all of $\RR^{N}$, the well-posedness of $h_{\nu}$ as the unique viscosity solution of \eqref{e.planmet2} only holds true in $\mathcal{H}_{\nu}$. Theorem \ref{t.planarint} shows that viscosity solutions of the PDEs on the left side of \eqref{e.planmet2} converge locally uniformly to the viscosity solution of the PDE on the right. A stochastic (and possibly viscous) version of these equations (termed the ``planar metric problem'') in $\RR^{N}$-stationary and finite range of dependent media (essentially independent and identically distributed media) was introduced by Armstrong and Cardaliaguet \cite{scottpierre} and studied by others \cite{FS, feldman} in the context of stochastic homogenization of geometric flows. In these works, they prove a similar result holds true almost surely. In particular, we highlight that the work of Feldman and Souganidis \cite[Proposition 3.6]{FS} implies the statement of Theorem \ref{t.planarint} in the case when $K\subseteq \mathcal{H}_{\nu}$. We will discuss related work and alternative approaches to what we have taken here in Section \ref{ss.discuss}.
 
 Finally, we add that our argument for proving Theorem \ref{t.planarint} in fact yields a more general homogenization result which is not covered by \cite{FS}, where they assumed $\mathbb{Z}^{N}$-stationarity and finite range of dependence. This assumption is not valid, for instance, in almost periodic media, which is characterized by large-scale correlations. Our proof, on the other hand, is sufficiently robust to cover the almost periodic setting:
 \begin{theorem}\label{t.ap}
 Suppose \eqref{h.nondeg}-\eqref{h.smootha} hold, and $a:\RR^{N}\rightarrow \RR$ is a Bohr almost periodic function (see Definition \ref{d.ap1} or Definition \ref{d.ap2}). For $h_{\nu}$ defined by \eqref{e.hnudefintro}, there exists a unique $c(\nu)\in [\sqrt{\theta}, \sqrt{\Theta}]$ such that for all $K\subseteq \RR^{N}$ compact, we have
    \begin{align*}
\lim_{T\to \infty} \sup_{x\in K}\left|\frac{1}{T}h_\nu(Tx)-c(\nu)(x\cdot \nu)\right|=0,
\end{align*}
and $c(\nu)=c(-\nu)$. 

 \end{theorem}

While we have stated Theorem \ref{t.planarint} and Theorem \ref{t.ap} in terms of the signed distance function, we note that our approach is also valid in proving homogenization results for the family of functions $\left\{u_{\nu}^{T}\right\}$ as $T\to \infty$, where $u_{\nu}^{T}: \mathcal{H}_{\nu}\rightarrow \RR$ is the unique viscosity solution of 
\begin{equation*}
\begin{cases}
    H(\nabla u_{\nu}^{T}, Tx)=0&\text{in $\mathcal{H}_{\nu}$,}\\
    u_{\nu}^{T}=0&\text{on $\Sigma_{\nu}$,}
    \end{cases}
\end{equation*}
whenever the Hamiltonian $H$ satisfies the following:
\begin{itemize}
    \item $H(\cdot, x)$ is convex, 1-homogeneous, and coercive, i.e. 
    \begin{equation*}
        \lim_{R\to\infty} \inf\left\{H(p,y): |p|\geq R, x\in \RR^{N}\right\}=+\infty. 
    \end{equation*}
    \item $H(p,\cdot)$ is Lipschitz continuous and periodic \emph{or} almost periodic. 
\end{itemize} 
Furthermore, if $u^{T}_{\nu}$ is in fact defined on all of $\RR^{N}$, and $u^{T}_{\nu}=-u^{T}_{-\nu}$, then one can obtain a homogenization result on all of $\RR^{N},$ using the same arguments as in the proofs of Theorem \ref{t.planarint} and Theorem \ref{t.ap}. 


\subsection{Motivation with Connections to the Homogeneous Modica-Mortola Problem}


The homogeneous case of $a \equiv 1$ reduces to the famous problem in phase transitions and the calculus of variations. The resulting homogeneous cell problem in the cell formula \eqref{e.sigma} for $\sigma$ is 
\begin{equation}\label{e.cell-a=1}
 \inf \left\{\int_{TQ_\nu} \left[ W(u) +\frac{1}{2}|\nabla u|^2\right]  \,dy : u \in \mathcal{C}(TQ_{\nu}) \right\}.\end{equation}
A classical argument using algebraic manipulations, made famous in the work of Modica and Mortola \cite{ModMort}, yields that as $T\to\infty$, the minimizer is the solution to  
\begin{equation}\label{e.inteikonal}
\begin{cases}
\frac{1}{2}|\nabla u|^{2}= W(u)\quad \text{in $\mathbb{R}^{N}$,} \\
u(y)\rightarrow \pm 1~\text{, as $y\cdot \nu\rightarrow \pm \infty$}.
\end{cases}
\end{equation}
It is clear that \eqref{e.inteikonal} encodes equipartition of energy between the gradient singular perturbation term and the potential term in the energy, see \eqref{e.cell-a=1}. The optimal $u$ for  \eqref{e.cell-a=1} is thus given by 
\[ u (y)  = q \circ (y\cdot \nu),\] where $q$ satisfies 
the associated Euler-Lagrange ODE, 
\begin{align} \label{e.heteroclinicint}
q^{\prime \prime} = W^\prime(q) \quad \quad \lim_{z \to \pm \infty} q(z) = \pm 1,
\end{align}
which is translation-invariant. Associated to this continuous symmetry, Noether's theorem yields a conservation law. This can be more simply derived, by multiplying \eqref{e.heteroclinicint} by $q^\prime,$ and integrating. We obtain ${|q^\prime|^2} = 2 W(q),$ a relation which dictates equipartition of energy. The solution  to (\ref{e.heteroclinicint}) is $q(z)=\tanh (\sqrt{2}z)$. 
 Note that  $y \cdot \nu$ is precisely the signed distance to the minimizing interface  $\Sigma_\nu = \{x \in \RR^N: x \cdot \nu = 0\}$.  The resulting surface tension is the constant 
 \begin{align*}
\sigma = \sqrt{2}\int_{-1}^1 \sqrt{W(s)}\,ds.
\end{align*}

Turning to the minimizer of the inhomogeneous cell problem 
\begin{equation}\label{e.cell-a}
 \inf \left\{\int_{TQ_\nu} \left[ a(y) W(u) +\frac{1}{2}|\nabla u|^2\right]  \,dy : u \in \mathcal{C}(TQ_{\nu}) \right\},\end{equation}
one could expect that  $q \circ (y\cdot \nu)$ is simply replaced with $q \circ h_\nu (y)$, where the inhomogeneous distance function $h_\nu$ is defined by (\ref{e.hnudefintro}). Indeed, we see that by definition of $h_{\nu}$, we have 
\begin{align*}
\frac{1}{2}|\nabla q(h_{\nu}(y))|^{2}=\frac{1}{2} (q^\prime(h_\nu(y))^{2} |\nabla h_\nu(y)|^2 = a(y) W(q(h_\nu(y)),  
\end{align*}
in which case $q\circ h_{\nu}$ achieves equipartition of energy. With this in hand, the cell formula  \eqref{e.sigma} for $\sigma$ (assuming the limit exists) would take the form  
\begin{equation}\label{e.cell-?}
\sigma(\nu) \, = \, \lim_{T \to \infty} \frac{1}{T^{N-1}} \int_{TQ_\nu}\left[a(y) W(q \circ h_\nu)+\frac{1}{2}|\nabla (q \circ h_\nu)|^2\right]  \,dy. \end{equation}
That is, in Theorem \ref{t.sigmarep}, we might have $\la_{0}(\nu)=0$ with $\sigma (\nu) = 
\underline{\lambda}(\nu) = \overline{\lambda}(\nu)$. 
This is false, at least for rational directions $\nu$, and we address why in the following subsection. Moreover, it is never the case that $q \circ h_\nu$ is a minimizer of the inhomogeneous cell problem 
(\ref{e.cell-a}) for any $T<\infty$. 
What Theorem \ref{t.sigmarep} shows, however,  is that this simple explicit formula yields upper and lower bounds for $\sigma$. Moreover, we make the case that on large scales as $T \to \infty$, the minimizer of 
(\ref{e.cell-a}) is close to $q \circ h_\nu$ (see Proposition \ref{p.wt} in comparison with \eqref{e.metriccom}).

 In the theory of homogenization, questions about bounds on the effective coefficients have a rather long and rich history in the context of optimal design (see \cite{Allaire,KohnEd,LurieCherkaev,Miltonbook,KohnMilton,MuratTartar} among many, many other references). This body of literature (largely) deals with effective bounds on linear elliptic (systems) of PDE using the homogenization method. Closer in spirit to our work is the paper \cite{BhattaKohn}, where the authors use nonlinear homogenization to study the so-called ``shape memory effect'' in polycrystals: the viewpoint there being that the heterogeneities in the texture field of the polycrystal within a nonconvex mesoscopic variational theory, upon nonlinear homogenization, yields a macroscopic theory whose global minimizers are recoverable strains. This coarse-graining procedure yields valuable bounds on the possible recoverable strains of the polycrystal-- information that is not directly accessible from the mesoscopic theory.  

\subsection{Outline of the Proofs and Discussion}\label{ss.discuss}
Here we outline the proof of Theorem \ref{t.sigmarep}. We then discuss whether or not it is possible for $\lambda_0(\nu)=0$ and $\sigma (\nu) = 
\underline{\lambda}(\nu)$, and this leads us to a discussion of Theorem \ref{t.planarint}. 

The upper bound in Theorem \ref{t.sigmarep} is more or less immediate: it essentially comes from energy comparison. We do need to alter the boundary conditions in the cell formula \eqref{e.sigma}, and this is achieved by the standard De Giorgi slicing technique, see Appendix \ref{s.app}. 
This procedure yields that for any $\nu \in \mathbb{S}^{N-1},$ the surface tension $\sigma(\nu),$ defined in \eqref{e.sigma}, has the alternative representation given by
\begin{align} \label{e.newbcint}
    \sigma(\nu) = \lim_{T \to \infty} \frac{1}{T^{N-1}} \inf \left\{ \int_{TQ_\nu} \left[ a(y) W(u) + \frac{1}{2}|\nabla u|^2\right]\,dy : \text{$u = q \circ h_\nu$ on $\partial (TQ_\nu)$}\right\}. 
\end{align}
Having proven the upper bound in Theorem \ref{t.sigmarep}, we turn to the lower bound. By the Direct Method in the Calculus of Variations, for each fixed $T$ there is a minimizer of the problem inside, which we denote by $u_T$. In other words, 
\begin{equation*}
    u_{T}:=\mathrm{argmin} \left\{ \frac{1}{T^{N-1}}\int_{TQ_\nu} \left[ a(y) W(u) + \frac{1}{2}|\nabla u|^2\right]\,dy : \text{$u = q \circ h_\nu$ on $\partial (TQ_{\nu})$}\right\}.
\end{equation*}
Define 
\begin{align*}
    \phi(z) := \sqrt{2}\int_0^z \sqrt{W(s)}\,ds, \quad \quad z \in \RR.
\end{align*}
Consider the function $h_\nu$ introduced in \eqref{e.hnudefintro}. It is easily shown (see Lemma \ref{l.hlips} below) that 
\begin{align*}
    |\nabla h_\nu(y) | = \sqrt{a(y)}, \quad \quad \text{for a.e. $y \in \RR^N$}.
\end{align*}
For any $T \gg 1,$ completing squares we find
\begin{equation} \label{e.3}
    \begin{aligned}
   & \frac{1}{T^{N-1}}  \int_{TQ_\nu} \left[ a(y) W(u_T) + \frac{1}{2}|\nabla u_T|^2\right]\,dy \\& \quad = \frac{\sqrt{2}}{T^{N-1}}\int_{TQ_\nu} \nabla h_\nu \cdot \sqrt{W(u_T)}\nabla u_T\,dy + \frac{1}{T^{N-1}} \int_{TQ_\nu}\left| \frac{\nabla u_T}{\sqrt{2}} - \sqrt{W(u_T)}\nabla h_\nu\right|^2\\
   &\quad = \frac{1}{T^{N-1}}\int_{TQ_\nu} \nabla h_\nu \cdot \nabla (\phi(u_T))\,dy + \frac{1}{T^{N-1}} \int_{TQ_\nu}\left| \frac{\nabla u_T}{\sqrt{2}} - \sqrt{W(u_T)}\nabla h_\nu\right|^2\\
   &\quad \geq \frac{1}{T^{N-1}} \int_{TQ_\nu} \nabla h_\nu \cdot \nabla (\phi(u_T))\,dy \\
   &\quad = \frac{1}{T^{N-1}}\int_{TQ_\nu} \nabla h_\nu \cdot \nabla (\phi(q \circ h_\nu)  )\,dy + \frac{1}{T^{N-1}}\int_{TQ_\nu} \nabla h_\nu \cdot \nabla \left( \phi(u_T) - \phi(q \circ h_\nu)  \right)\,dy\\
   &\quad = \frac{1}{T^{N-1}} \int_{TQ_\nu} \left[a(y) W(q \circ h_\nu) + \frac{1}{2}|\nabla (q \circ h_\nu)|^2 \right]\,dy \\ & \quad \quad +  \frac{1}{T^{N-1}}\int_{TQ_\nu} \nabla h_\nu \cdot \nabla \left( \phi(u_T) - \phi(q \circ h_\nu)  \right)\,dy,
       \end{aligned}
\end{equation}
where, in the last line, we used the fact that $q \circ h_\nu$ achieves equipartition of energy, while completing squares one more time. Defining 
\begin{equation*}
\begin{aligned}
    \overline{\lambda}(\nu) := \limsup_{T \to \infty} \frac{1}{T^{N-1}} \int_{TQ_\nu} \left[a(y) W(q \circ h_\nu) + \frac{1}{2} |\nabla (q \circ h_\nu)|^2 \right]\,dy, \\
    \underline{\lambda}(\nu) := \liminf_{T \to \infty} \frac{1}{T^{N-1}} \int_{TQ_\nu} \left[a(y) W(q \circ h_\nu) + \frac{1}{2} |\nabla (q \circ h_\nu)|^2 \right]\,dy,
\end{aligned}
\end{equation*}
the proof of the lower bound in Theorem \ref{t.sigmarep} (specifically \eqref{e.sigrep}) is now immediate, provided we can show that
\begin{align} \label{e.lambda0}
    \limsup_{T \to \infty} \frac{1}{T^{N-1}}\left|\int_{TQ_\nu} \nabla h_\nu(y) \cdot \left( \phi(u_T) - \phi(q \circ h_\nu)\right)\,dy \right| =: \lambda_0(\nu)\leq \La_{0},
\end{align}
for some $\Lambda_0 > 0.$ Obtaining this error bound is complicated by the fact that $\la_{0}(\nu)$ couples oscillations and concentrations. In Section \ref{s.conc}, we present novel tools and concentration estimates in order to control $\la_{0}(\nu)$, which we briefly summarize here:
\begin{enumerate}
\item We recall that for each $T > 1,$ the function $u_T \in H^1(TQ_\nu)$ is a minimizer of the variational problem in \eqref{e.newbcint}. In Theorem \ref{l.step1} below, we show that, as $T \to \infty,$
 $\{u_T(T\,\cdot)\}_{T > 0}$ converges to $u_{0,\nu}$ \textit{strongly} in $L^1(Q_\nu)$ , where $u_{0,\nu}$ is given by 
\begin{align*}
    u_{0,\nu}(y) := \left\{
    \begin{array}{cc}
        1 & y \cdot\nu > 0, \\
        -1 & y \cdot \nu < 0.
    \end{array}
    \right.
\end{align*}
We emphasize that this convergence is not simply along a subsequence, since the limit is unique. This further implies that $\phi(q \circ h_\nu(T
\,\cdot))$ also converges in $L^1(Q_\nu)$ to the same limit, $u_{0,\nu}.$ It follows that 
\begin{equation}\label{e.phiconv}
\phi(u_T(T\,\cdot)) - \phi(q \circ h_\nu(T\,\cdot)) \to 0
\end{equation}
in $L^1(Q_\nu)$ as $ T \to \infty.$ 
\item For each $T \gg 1,$ writing $u_T(y) =: \tanh(\sqrt{2}w_T(y)),$ we show in Lemma \ref{l.bounds} that 
\begin{align}
    \label{l.compup}
    -1<  u_T(y) < 1,  \quad \quad y \in TQ_\nu.
\end{align}
We prove in Proposition \ref{p.wt} that there exist positive numbers $\alpha_0, \eta_0$ depending only on $\nu,\theta$ and $\Theta,$ such that
\begin{equation}\begin{aligned}
    \label{l.complow}
 \sqrt{\Theta}(y\cdot \nu) - \alpha_0 \geqslant    w_T(y) \geqslant \sqrt{\theta}(y \cdot \nu) - \eta_0 \quad \mbox{ if } w_T > 0,\\
-\sqrt{\theta}(y\cdot \nu) + \eta_0 \geqslant     w_T(y) \geqslant -\sqrt{\Theta}(y \cdot \nu) + \alpha_0 \quad \mbox{ if } w_T(y) < 0.
\end{aligned}\end{equation}
In particular, this shows that the profiles $u_T$ and $q \circ h_\nu,$ in blown up variables as in (a) above, converge to the sharp interface limit at the same rate (see \eqref{e.metriccom}). 
\end{enumerate}

Naturally, the above mentioned results are proven in order to pass to the limit in the error term in \eqref{e.3}. It is easy to see that the topology of convergence in item (1) is not strong enough to conclude that the asymptotic contribution of this error term vanishes. 
However, we are able to put these ingredients together, to obtain bounds on the error term in \eqref{e.3}. This is carried out in Section \ref{s.errorbound}.  

These results beg the following natural questions: 
\begin{itemize}
\item Can the error term $\lambda_0(\nu)=0$, with $\sigma (\nu) = 
\underline{\lambda}(\nu) = \overline{\lambda}(\nu)$?
\item Is it ever the case that the minimizer $u_T$ to the cell problem  is simply $q \circ h_\nu$?
\end{itemize}
The answer to the second question is no unless $a$ is constant. The same is true for the first question in the case of rational directions $\nu$ (i.e., when $\nu$ has rational components). In fact, when $\nu\in \mathbb{Q}^{N}\cap \mathbb{S}^{N-1}$, William Feldman and Peter Morfe \cite{willpeter} have recently shown the authors the following: if
(\ref{e.cell-?}) holds true in a rational direction $\nu$, then one can argue that $q \circ h_\nu$ must be a minimizer on an infinite strip. An analysis of the Euler-Lagrange equations readily leads to a contradiction in that $h_{\nu}$ must be harmonic (which is true if and only if $a$ is constant). {We are unaware of an argument, but in fact, it is perhaps the case that if $\lambda_0(\nu) = 0$ for any $\nu \in \mathbb{S}^{N-1},$ then $a$ must be constant. Nevertheless, for any $\nu\in \mathbb{S}^{N-1}$ such that $\la_{0}(\nu)\neq 0$, this implies that, surprisingly, equipartition of energy does not hold.}

An interesting question, which we are unable to resolve here, is the following: {for a given choice of periodic heterogeneities $a,$ is at least one of the bounds in Theorem \ref{t.sigmarep} close to being sharp? Naturally, verifying the sharpness of the lower bound in Theorem \ref{t.sigmarep} requires passing to the limit in the term $\lambda_0(\nu);$ this in turn requires convergence of $\{\nabla h_\nu(T_m\, \cdot)\}_m$ in suitable topologies.} We provide partial progress in this direction with the second main contribution of this paper in Theorem \ref{t.planarint}:  we show that there exists a unique $c(\nu)\in [\sqrt{\theta}, \sqrt{\Theta}]$ such that , for any sequence $\{T_m\}_{m \in \mathbb{N}}$ tending to infinity, for the functions $k_m(\cdot) := T_{m}^{-1} h_\nu(T_m\,\cdot),$
\begin{align*}
    \{k_m\}_m \mbox{ converges locally uniformly in $\RR^{N}$ to } x \mapsto c(\nu)x \cdot \nu. 
\end{align*}

We prove Theorem \ref{t.planarint} in two steps. In Lemma \ref{l.hompmp}, we first show that for every sequence $\left\{T_{m}\right\}$ tending to infinity, there exists a subsequence and a function $c(\nu)\in [\sqrt{\theta}, \sqrt{\Theta}]$ such that $k_{m}(x)\to c(\nu)x\cdot \nu$ locally uniformly in $\RR^{N}$. The proof of Lemma \ref{l.hompmp} uses various properties of Bohr almost periodic functions and ideas which come from the proof of the Stone-Weierstrauss theorem. Upon establishing Lemma \ref{l.hompmp}, we then argue that $c(\nu)$ must be unique in order to establish convergence of the full sequence.

Our uniqueness argument relies on the existence of periodic correctors in the setting of periodic Hamilton-Jacobi equations (see Theorem \ref{t.lpv}). It was pointed out to us that an alternative approach to proving that 
\begin{align*}
    \{k_m\}_m \mbox{ converges locally uniformly in $\mathcal{H}_{\nu}$ to } x \mapsto c(\nu)x \cdot \nu, 
\end{align*}
could use the existence of periodic correctors, a comparison principle on half-spaces (stated in \cite{FS}, without proof), and the perturbed test function method of Evans \cite{evans}. While this argument may appear more direct to specialists in homogenization of Hamilton-Jacobi equations, we highlight that aside from the existence of periodic correctors, the proof of Theorem \ref{t.planarint} is entirely self-contained. Our results concerning the large scale behavior for $h_\nu$ are contained in Section \ref{s.pmp}.

\section{Basic Properties of \texorpdfstring{$h_{\nu}$}{} and Equipartition of Energy}

\subsection{Existence and Basic Properties of \texorpdfstring{$h_{\nu}$}{}}\label{ss.hnu} We introduce a Riemannian metric in $\RR^N,$ which is conformal to the standard Euclidean one. To be precise, given a Lipschitz curve $\gamma : [0,1] \to \RR^N$, we define its length to be
\begin{align*}
L(\gamma) := \int_0^1 \sqrt{a(\gamma(t))}|\dot{\gamma}(t)|\,dt.
\end{align*}
Naturally, $L(\gamma)$ does not depend on the parametrization of $\gamma.$ We define the distance between points $y_1,y_2 \in \RR^N$ in the $\sqrt{a}-$metric, by 
\begin{align}
\label{e.dist}
d_{\sqrt{a}}(y_1,y_2) := \inf_{\gamma(0) = y_1,\gamma(1) = y_2} L(\gamma).  
\end{align}
{The existence of a minimizer, i.e., a geodesic in \eqref{e.dist}, and its regularity, follow by classical arguments via the Hopf-Rinow theorem, since $a$ is bounded away from zero by \eqref{h.nondeg} (for details, see \cite[Lemma 2.9]{S88}), thereby rendering $\RR^N$ geodesically complete. } 
\par Let $\nu \in \mathbb{S}^{N-1},$ set $\Sigma_\nu := \{x \in \RR^N : x \cdot \nu = 0\},$ and define $h_\nu : \RR^N \to \RR$ by 
\begin{align} \label{e.sgddist}
h_\nu(x) := \left\{ 
\begin{array}{cc}
d_{\sqrt{a}}(x, \Sigma_\nu) & \mbox{ if } x \cdot \nu \geqslant 0, \\
-d_{\sqrt{a}}(x, \Sigma_\nu) & \mbox{ if } x \cdot \nu < 0.
\end{array}
\right.
\end{align}
The function $h_{\nu}(x)$ represents a signed distance function from $x$ to the plane $\Sigma_{\nu}$. 
\begin{remark}\label{r.hodd}
Observe that, by \eqref{e.sgddist}, and since $\Sigma_{\nu}=\Sigma_{-\nu}$, we have
\begin{equation*}
h_{\nu}(x)=\begin{cases} d_{\sqrt{a}}(x, \Sigma_{\nu}), &x\cdot \nu \geqslant  0\\
-d_{\sqrt{a}}(x, \Sigma_{\nu}), &x\cdot \nu < 0,
\end{cases}
\quad\text{and}\quad 
h_{-\nu}(x)=\begin{cases} d_{\sqrt{a}}(x, \Sigma_{\nu}), &x\cdot (-\nu) \geqslant  0,\\
-d_{\sqrt{a}}(x, \Sigma_{\nu}), &x\cdot (-\nu) < 0,
\end{cases}
\end{equation*}
which imply that
\begin{equation*}\label{e.hodd}
h_{-\nu}(x)=-h_{\nu}(x). 
\end{equation*}
In particular, $h_{\nu}$ is odd with respect to $\nu$. 
As $h_{\nu}$ is a type of signed distance, it in fact satisfies an Eikonal equation.

\end{remark}
\begin{lemma}
\label{l.hlips}
 The function $h_\nu$ is Lipschitz continuous in $\RR^N,$ with 
\begin{align*}
|\nabla h_\nu(x)| = \sqrt{a(x)} \quad \mbox{for a.e. } x \in \RR^N.
\end{align*}
\end{lemma}
\begin{proof} 
See \cite[Lemma 11]{S88}.
\end{proof}

Since $|\nabla{h_{\nu}}|=\sqrt{a}\in [\sqrt{\theta}, \sqrt{\Theta}]$ by \eqref{h.nondeg}, by \eqref{e.dist} and \eqref{e.sgddist} we have
\begin{equation}\label{e.metriccom}
\begin{cases}
\sqrt{\theta}|x\cdot \nu|\leqslant h_\nu(x)\leqslant \sqrt{\Theta}|x\cdot \nu|&\text{if $x\cdot \nu\geqslant  0$,}\\
-\sqrt{\Theta}|x\cdot \nu|\leqslant h_{\nu}(x)\leqslant -\sqrt{\theta}|x\cdot \nu|&\text{if $x\cdot \nu<0$.}
\end{cases}
\end{equation}
\subsection{Equipartition of Energy: \texorpdfstring{$|\nabla u|^2 = 2 a(x) W(u)$}{}} \label{s.eikonal}
In this section we use the Riemannian geometry framework introduced above to find approximate ``one-dimensional'' solutions to the degenerate Eikonal equation 
\begin{align} \label{e.eik}
\frac{|\nabla u |^2}{2} = a(x) W(u)
\end{align}
in large cubes in $\RR^N,$ in a sense to be made precise. This analysis is crucial in the proof of Theorem \ref{t.sigmarep}. 
Taking inspiration from the cell formula \eqref{e.f0def}, for $\nu \in \mathbb{S}^{N-1}$, we seek solutions $u$ to \eqref{e.eik} that ``connect'' the zeroes of $W,$ i.e., $u(x) \to \pm 1$ as $ x\cdot \nu \to \pm \infty.$ 
Consider the ansatz
\begin{align*}
u(x) := (q\circ \dstf)(x), 
\end{align*}
for some $q:\mathbb{R}\rightarrow \RR$ to be determined. Inserting this into \eqref{e.eik}, we obtain 
\begin{align*}
\frac{1}{2} (q^\prime(h_\nu(x))^{2} |\nabla h_\nu(x)|^2 = a(x) W(q(h_\nu(x)). 
\end{align*}
As $|\nabla h_{\nu}| = \sqrt{a}$ pointwise a.e. (see Lemma \ref{l.hlips}), the function $q$ must satisfy the ordinary differential equation 
\begin{align} \label{e.hetero}
q^\prime = \sqrt{2}\sqrt{W(q)}. 
\end{align}
By \eqref{e.metriccom}, we see that $h(x)\rightarrow \pm\infty$ as $x\cdot \nu\rightarrow \pm\infty$. In particular, to connect the zeros of $u$ at $\pm\infty$, we require that $q(z) \to \pm 1 $ as $z \to \pm \infty.$ In order to identify this function $q$, we consider a suitable initial condition associated to \eqref{e.hetero} in Proposition \ref{p.qprop}. 

For convenience, we recall some basic properties of the hyperbolic tangent and secant functions, $\tanh$ and $\sech$, respectively, which will be used throughout the rest of the paper: 
\begin{equation}\label{e.tanhsech}
\begin{cases}
\text{$\tanh(x)=\frac{e^{x}-e^{-x}}{e^{x}+e^{-x}}$ is an odd function,}\\
\text{$-1< \tanh(x)< 1$, for all $x \in \RR,$}\\
\text{There exists $c_{1}, c_{2}>0$ such that $\begin{cases}\lim_{x \to \infty} \left|1-\tanh(x)\right|\leq c_{1}e^{-c_{2}|x|},\\ \lim_{x \to -\infty} \left|-1-\tanh(x)\right|\leq c_{1}e^{-c_{2}|x|},\end{cases}$}\\
1-\tanh^{2}(x)=\sech^{2}(x), \text{for all $x \in \RR$,}\\
|\sech(x)|=\left|2\frac{e^{x}}{e^{2x}+1}\right|\leqslant 2e^{-|x|}\text{ is an even function, and $0\leqslant \sech(x)\leqslant 1,  \forall x \in \RR$} ,\\
(\sech(x))^\prime = -\tanh (x) \sech(x), (\tanh(x))^\prime=\sech^{2}(x),  \forall x \in \RR, \\
\text{$\sech^{4}(x)$ is decreasing on $(0,\infty)$}.
\end{cases}
\end{equation}

\begin{proposition}\label{p.qprop}
There exists a unique solution to 
\begin{equation}\label{e.q}
q^\prime = \sqrt{2}\sqrt{W(q)}, \quad q(0)=0.
\end{equation}
Moreover, there exist $c_{1}, c_{2}>0$ such that 
\begin{equation}\label{e.expbnds}
\begin{cases}
q(z)  \geqslant 1-c_1 e^{-c_2|z|} & \mbox{ if } z > 0 , \\
q(z)  \leqslant -1+c_1 e^{-c_2|z|} & \mbox{ if } z < 0. 
\end{cases}
\end{equation}
In particular, $q(z) \to \pm 1 $ as $z \to \pm \infty.$
\end{proposition} 

\begin{proof}
It is easy to see that $q(z):=\tanh (\sqrt{2}z)$ is the unique solution to \eqref{e.q}. The exponential bounds \eqref{e.expbnds}  immediately follow from \eqref{e.tanhsech}.
\end{proof}

\section{Properties of minimizers in the cell problem} \label{s.conc}
By Lemma \ref{l.newbc}, we have
\begin{align}\label{e.tvar}
\sigma(\nu)&=  \lim_{T \to \infty} \frac{1}{T^{N-1}} \inf \Big\{\int_{TQ_\nu} \left[ a(y)W(u)+ \frac{1}{2}|\nabla u|^2  \right] \,dy : \notag\\ &\qquad u \in H^1(TQ_\nu), u|_{\partial TQ_\nu} = q \circ h_{\nu} \Big\}\notag\\
& =\lim_{T \to \infty} \inf \Big\{\int_{Q_\nu} \left[ Ta(Tx)W(V)+ \frac{1}{T}\frac{|\nabla V|^2}{2}  \right] \,dx : \notag\\ &\qquad V \in H^1(Q_\nu), V|_{\partial Q_\nu} = q \circ h_{\nu}(T\cdot) \Big\} .
\end{align}
In the remainder of this section, we suppress the subscript $m$ for notational ease, with the understanding that when we let $T \to \infty$ in the end, we do so along this particular subsequence $T_m \to \infty.$ 
 
We introduce the function $v_{T}\in H^{1}(Q_{\nu})$ satisfying
\begin{align}\label{e.vTdef}
v_T \in  \mathrm{argmin} \left\{ E_T(V) := \int_{Q_\nu} \left[ Ta(Tx) W(V) + \frac{1}{T}\frac{|\nabla V|^2 }{2}\right]\,dx : V|_{\partial Q_\nu} = q \circ h_\nu(Tx)\right\}.
\end{align}
Since $q \circ h_\nu(T\,\cdot)$ is an admissible competitor in the variational problem \eqref{e.tvar}, we may assume that 
\begin{equation}\begin{aligned} \label{e.competitor}
&\int_{Q_\nu} \left[ Ta(Tx) W(v_T) + \frac{1}{T}\frac{|\nabla v_T|^2}{2} \right] \,dx \\
&\leqslant \int_{Q_\nu} \left[ Ta(Tx) W(q\circ h_{\nu}(Tx)) + \frac{1}{T}\frac{|\nabla q\circ h_{\nu}(Tx)|^2 }{2} \right] \,dx\\
&\leqslant O(1)
\end{aligned}\end{equation}
as $T \to \infty.$

\begin{lemma}\label{l.step1}
Let $v_{T}: Q_{\nu}\rightarrow \RR$ satisfy \eqref{e.vTdef}. 
 There exists a subsequence, not relabeled, such that 
\begin{align}
\label{e.limcon}
v_T \to u_0 \quad in \quad L^1(Q_\nu), 
\end{align}
where, we recall, $u_{0}: \RR^{N}\rightarrow \RR$ is defined by
\begin{align*}
u_0(x) := \left\{
\begin{array}{cc}
1 &  x \cdot \nu > 0,\\
-1 & x \cdot \nu < 0.
\end{array}
\right.
\end{align*}  
\end{lemma}

\begin{proof}
Since $v_{T}$ satisfies \eqref{e.vTdef}, it verifies the uniform energy bound \eqref{e.competitor}.
As $a$ is bounded away from zero, this estimate yields, via a standard compactness argument using the Modica-Mortola inequality,  that $\{v_T\}$ is precompact in $L^1(Q_\nu)$ (see \cite{FT89} or \cite{S88}). Let $U$ be an $L^1$ cluster point of $\{v_T\}_T.$ By \eqref{e.tvar}, the energies of the minimizers $v_T$ converge to $\sigma(\nu).$

We recall that $\sigma(\nu)$ is the limiting energy corresponding to $u_{0}$, and we claim that $U = u_0$.  We extend $v_T$ to all of $\RR^N$ by setting $v_T(x) := q\circ h_\nu(Tx)$ for $x \not \in Q_\nu$ and, likewise, we extend $U$ to all of $\RR^N$ by setting $U = u_0$ outside $Q_\nu.$  Let $\tau > 0$ be fixed, and we work in the dilated cube $(1 + \tau)Q_\nu.$ We label the restrictions of $v_T$ and $U$ to $(1 +\tau)Q_\nu$ by $\tilde{v}_T, \tilde{U},$ respectively. By \cite[Theorem 5.8]{EG}, $\tilde{U} \in BV((1+\tau)Q_\nu).$ We note that by the aforementioned compactness arguments,
\begin{align}
\label{e.taucompactness}
\tilde{v}_T \to \tilde{U} \quad \mbox{ in } L^1((1+\tau)Q_\nu) \mbox{ as } T \to \infty,\\ \label{e.taucompactness2}
D\tilde{v}_T \rightharpoonup D\tilde{U} \quad \mbox{ weakly-* in the sense of measures as $T\rightarrow \infty$,}
\end{align}  
for all $\tau > 0.$  
Since $U$ is piecewise constant, $\nabla \tilde{U} = 0,$ and   
we find that $dD \tilde{U}= 2\nu_{\tilde{U}}\,d\sh^{N-1} \lfloor J_{\tilde{U}}, $ where $J_{\tilde{U}}$ is the jump set of $\tilde{U}$ and $\nu_{\tilde{U}} = \frac{\,d\,D\tilde{U}}{\,d\,|D\tilde{U}|} $ on $J_{\tilde{U}}$ (see \cite{maggibook}). 
We claim that 
\begin{equation}\label{e.bvstep}
\int_{J_{\tilde{U}} \cap \overline{Q}_\nu} \,d D \tilde{U} = 2\nu . 
\end{equation}
By \eqref{e.taucompactness2}, for every $\phi \in C_c((1 + \tau)Q_\nu)$ and for every unit vector $e \in \mathbb{S}^{N-1},$  we have 
\begin{align} \label{e.weakconvergence}
\int_{(1+\tau)Q_\nu} \phi e \cdot \nabla \tilde{v}_{T} \,dx = \int_{(1+\tau)Q_\nu} \phi e \cdot dD\tilde{v}_T  \to  \int_{(1+\tau)Q_\nu} \phi e \cdot dD\tilde{U}  \quad \quad \mbox{ as } T \to \infty.   
\end{align} 

In particular, let $\phi \in C_c^\infty((1+\tau)Q_\nu)$ be such that $\phi \equiv 1 $ on $\overline{Q}_\nu,$ $0 \leqslant \phi \leqslant 1,$ and $\phi \equiv 0$ on $(1+\tau)Q_\nu \backslash (1 + \tau/2) Q_\nu.$ If $e \in \{\nu_1,\cdots, \nu_N\},$ we then have that as $T\rightarrow \infty$, 
\begin{align*}
\int_{(1+\tau/2)Q_\nu\backslash  Q_\nu } \phi e \cdot \nabla \tilde{v}_T \,dx +   \int_{Q_\nu} e \cdot \nabla \tilde{v}_T\,dx \to \int_{(1+\tau/2)Q_\nu \backslash \overline{Q}_\nu} \phi e \cdot D\tilde{U}+ \int_{\overline{Q}_\nu} e \cdot \,dD\tilde{U} . 
\end{align*}
As $\tilde{v}_T \equiv q_T \circ h_\nu(T\, \cdot)$ and $\tilde{U} = u_0$ outside $Q_\nu$, we find that the first and the third terms in the previous display are $O(\tau^{N-1}).$ It remains to evaluate the limit of the second term as $T \to \infty$ . With the choice $e = \nu_N = \nu$, by the fundamental theorem of Calculus, we find that, as $T \to \infty,$ 
\begin{align*}
\int_{Q_\nu} \nu \cdot \nabla \tilde{v}_T \,dx \to 2, \quad \quad \mbox{ as } T \to \infty,
\end{align*} 
because $q_\nu \circ h_\nu(T\,\cdot)$ is exponentially close to $1$ and $-1$ respectively, on the top and bottom faces of $Q_\nu,$ i.e., $\{x \in \overline{Q}_\nu: x \cdot \nu = \pm \frac{1}{2}\}. $
It follows that 
\begin{align*}
\int_{\overline{Q}_\nu} \nu \cdot \,d D\tilde{U} = 2 + O(\tau^{N-1}).
\end{align*}
Finally, for the lateral directions $e = \nu_1,\cdots, \nu_{N-1},$ we have, 
\begin{align*}
\int_{Q_\nu} e \cdot\, D\tilde{v}_T \,dx &= \int_{Q_{\nu}\cap \left\{x \cdot e = \frac{1}{2}\right\}} q\circ h_\nu(Tx)\,d \sh^{N-1}(x) \\
&\quad \quad \quad \quad \quad - \int_{Q_{\nu}\cap \left\{x \cdot e  = -\frac{1}{2}\right\}} q\circ h_\nu(Tx) \,d \sh^{N-1}(x)  \\ & \xrightarrow[T\to \infty]{} \int_{Q_{\nu}\cap \left\{x \cdot e = \frac{1}{2}\right\}} u_0\,d \sh^{N-1}(x) - \int_{Q_{\nu}\cap \left\{x \cdot e = -\frac{1}{2}\right\}} u_0 \,d \sh^{N-1}(x)= 0,
\end{align*}
We deduce that 
\begin{align} \label{e.averagetau}\int_{\overline{Q}_\nu} \,d D\tilde{U} = \sum_{i=1}^N \left( \int_{\overline{Q}_\nu} \,d D\tilde{U}\cdot \nu_i \right)\nu_i = 2\nu + O(\tau^{N-1}).
\end{align}
Since $D\tilde{U} = 2 \sh^{N-1}\lfloor J_{\tilde{U}},$ it follows that 
\begin{align} \label{e.jutau}
J_{\tilde{U}}\cap \overline{Q}_\nu = J_U \cup \{ x \in \partial Q_\nu : \mathrm{trace}(U)(x) \neq u_0(x)\} =: K_U.
\end{align}
and the set $K_U$ on the right hand side is independent of $\tau > 0.$ Indeed, note that the extension $\tilde{U}$ of $U$ does not depend on $\tau,$ and we now call it $U^0.$ The Radon-Nikodym derivative $\frac{\,dDU^0}{\,d|DU^0|}$ is equal to $\nu_U$ on $J_U,$ and it is equal to the normal to the boundary  of $Q_\nu$, $\nu_{\partial Q_\nu}$, on $K_U \backslash J_U.$ Now \eqref{e.averagetau} reduces to 
\begin{align*}
\int_{K_U} \,d DU^0 = 2\nu + O(\tau^{N-1}).
\end{align*}
Letting $\tau \to 0^+$ we deduce that 
\begin{align} \label{e.dude}
\int_{K_U} DU^0= 2\nu.
\end{align}
By Theorem \ref{t.fonseca} on $(1 + \tau)Q_\nu$ for each fixed $\tau,$ then sending $\tau \to 0^+$, and then using Jensen's inequality owing to the convexity of the one-homogeneous extension of $\sigma$, $\tilde{\sigma}$, we have
\begin{equation} \label{e.jensen1}\begin{aligned}
&\sigma(\nu) =\lim_{\tau \to 0^+} \liminf_{T \to \infty} E_T(\tilde{v}_T;(1+\tau)Q_\nu) \\ & \quad \quad \geqslant \limsup_{\tau \to 0} \int_{{J}_{\tilde{U}} \cap (1 + \tau)Q_\nu} {\sigma}\left( \,\nu_{\tilde{U}}\right)\,d \sh^{N-1} \geqslant \limsup_{\tau \to 0} \int_{{J}_{\tilde{U}} \cap \overline{Q}_\nu} {\sigma}\left( \,\nu_{\tilde{U}}\right)\,d \sh^{N-1}\\&\quad \quad = \int_{K_U} \sigma \left( \frac{\,d\,DU^0}{\,d\,|DU^0|}\right) \,d \sh^{N-1} \\
&\quad \quad =  \int_{K_U} \tilde\sigma \left( \frac{\,d \,DU^0}{\,d|\,DU^0|} \sh^{N-1}(K_U) \right)  \frac{\,d \sh^{N-1}}{\sh^{N-1}(K_U)}  \\
&\quad \quad \geqslant \tilde{\sigma}\left( \int_{K_U} \frac{\,d\,DU^0}{\,d |\,DU^0|} \sh^{N-1}(K_U) \frac{\, d\sh^{N-1}}{\sh^{N-1}(K_U)}\right). 
\end{aligned}
\end{equation}
But $|DU^0|\lfloor K_U = 2 \sh^{N-1} \lfloor K_U,$ and we find by the one-homogeneity of $\tilde{\sigma}$ that 
\begin{align} \label{e.jensen2}
\tilde{\sigma}\left( \int_{K_U} \frac{\,d\,DU^0}{\,d |\,DU^0|} \right)\sh^{N-1}(K_U) = \frac{1}{2}\tilde{\sigma}\left(\int_{K_U} \,d \,DU^0 \right).
\end{align}
Again using the one-homogeneity of $\tilde\sigma,$ the equality \eqref{e.dude} implies that the right hand side of \eqref{e.jensen2} evaluates to $\frac{1}{2}\tilde{\sigma}(2\nu) = \tilde{\sigma}(\nu) = \sigma(\nu).$ In turn, plugging this into the chain of inequalities in \eqref{e.jensen1}, we learn that we must have equalities throughout. But equality holds in Jensen if and only if $\frac{\,d\,DU^0}{\,d|\,DU^0|} \lfloor K_U $ is a constant. This immediately implies that $\sh^{N-1} \left( x \in \partial Q_\nu: \mathrm{trace}(U) \neq u_0\right) = 0,$ and thus that $U$ inherits the trace $u_0$ from the sequence $\{v_T\}.$ Furthermore, we conclude $K_U = J_U$ up to a set of $\sh^{N-1}$ null measure, and so, $U \equiv u_0$ in $Q_\nu,$ yielding \eqref{e.limcon}. 

\end{proof}

For what follows, we need finer, quantitative versions of the foregoing convergence result and, in particular, of the convergence of the functions $u_T.$ The remainder of this section is devoted to obtaining these estimates. The next preparatory lemma is an immediate consequence of the maximum principle. 
\begin{lemma} \label{l.bounds}
Let $u_T$ be a minimizer to \eqref{e.newbcint}. Then 
\begin{align*}
   -1 <  u_T(y) < 1, \quad y \in TQ_\nu.
\end{align*}
\end{lemma}
\begin{proof}
For each $T$, as $u_T$ is a minimizer of the energy 
\begin{align*}
    \int_{TQ_\nu} \left[ a(y) W(u) + \frac{1}{2}{|\nabla u|^2} \right]\,dy,
\end{align*}
subject to Dirichlet boundary conditions $u_T = q\circ h_\nu$ along $\partial (TQ_\nu)$, it follows by standard arguments that $u_T$ is a classical solution of the associated Euler-Lagrange equations 
\begin{equation}
\begin{cases}
    \Delta u = a(y) W^\prime(u)=-4a(y)u(1-u^{2}) & y \in TQ_\nu,\\
    u(y) = q \circ h_\nu(y), & y \in \partial (TQ_\nu).
    \end{cases}
\end{equation}
We know that for any $T < \infty,$ $\sup_{y \in \partial(TQ_\nu)} |q \circ h_\nu| < 1.$ Suppose, by way of contradiction, that there exists $y_0 \in TQ_\nu$ such that 
\begin{align*}
    u_T(y_0) = \max_{y \in \overline{TQ_\nu}} u_T(y)> 1.
\end{align*}
Then $\Delta u_T (y_0) \leqslant 0,$ while $W^\prime(u_T(y_0)) > 0,$ and $a(y_0)\geq \theta > 0$, yielding a contradiction. It follows that $u_T(y)\leq 1$ for every $y \in TQ_\nu.$ A similar argument shows that $u_T(y)\geq -1$ for every $y \in TQ_\nu.$ Finally, a standard argument (as in the proof of the strong maximum principle) using the Hopf lemma yields the desired strict inequalities. 
\end{proof}
Define 
\begin{align} \label{e.wtdef}
    w_T := \frac{1}{\sqrt{2}}\tanh^{-1} u_T.
\end{align}
By Lemma \ref{l.bounds}, $w_T: TQ_\nu \to (-\infty,\infty)$ is smooth. Further, $w_T$ is a classical solution to the PDE 
\begin{equation} \label{e.wtpde}
    \begin{cases}
    \Delta w_T(y) = \frac{4}{\sqrt{2}}\tanh (\sqrt{2}w_T(y)) \left( |\nabla w_T(y)|^2 - a(y) \right), &  y \in TQ_\nu,\\
    w_T(y) = h_\nu(y)& y\in \partial(TQ_{\nu}).
    \end{cases}
\end{equation}
In the remainder of this section we obtain fine properties of the function $w_T,$ specifically in Proposition \ref{p.wt} below. A crucial ingredient in the argument is the following result due to L. Caffarelli and A. Cordoba \cite[Theorem 2]{CC}. 
\begin{proposition} \label{l.cc}
Consider the functions $u_T:TQ_\nu \to \RR.$ 
Then, as $T \to \infty,$ for each $\mu\in (-1,1)$ the level sets $\{x \in TQ_\nu: u_T(x) = \mu\}$ are at a uniformly bounded distance from $\Sigma_\nu \cap TQ_\nu.$ To be precise, for each $\mu \in (-1,1)$ there exists a constant $\eta(\mu,\nu) > 0,$ only depending on $\mu$ and $\nu,$ and independent of $T \gg 1$, such that 
\begin{align}
    \{y \in TQ_\nu : u_T(x) = \mu \} \subset \{y \in TQ_\nu: |y\cdot \nu| < \eta(\mu,\nu)\}. 
\end{align}
\end{proposition}
Equipped with the foregoing proposition, we are ready to state the proof the main result of this section, namely, that the functions $w_T$ defined in \eqref{e.wtdef} are essentially linear. 
\begin{proposition}
\label{p.wt}
Let $w_T$ be as in \eqref{e.wtdef}, let $T \gg 1$, and define the constants $\eta_0 := \sqrt{\theta}\eta(0,\nu) > 0,$ and $\alpha_0 := \sqrt{\Theta}\eta(0,\nu) > 0,$ where $\eta(0,\nu)$ is obtained from Proposition \ref{l.cc} corresponding to the level set $\mu =0$. Then, for all $T \gg 1,$ the following hold: 
\begin{equation} \label{e.linearlbnd}
    \begin{aligned}
    &\sqrt{\Theta} (y\cdot \nu) - \alpha_0 \geq w_T(y) \geq \sqrt{\theta}(y \cdot \nu) - \eta_0 \quad &\mbox{ if } w_T (y) > 0, \\
    &- \sqrt{\theta} (y\cdot \nu) + \eta_0 \geq w_T(y) \geq - \sqrt{\Theta}(y \cdot \nu)+ \alpha_0 \quad &\mbox{ if } w_T(y) < 0. 
    \end{aligned}
\end{equation}
\end{proposition}
\begin{proof}
Owing to the continuity of $w_T,$ the sets 
\begin{align*}
    \Omega_\pm := \{y \in TQ_\nu : w_T(y) \gtrless 0 \}
\end{align*}
are open. We show the lower bound in the first statement in \eqref{e.linearlbnd}. Define the function $\zeta_T : \overline{\Omega_+} \to \RR$ by the formula
\begin{align*}
    \zeta_T(y) := \frac{y \cdot \nu}{w_T(y) + \eta_0}, \quad y \in \overline{\Omega_+}.
\end{align*}
Being a continuous function on the compact set $\overline{\Omega_\pm}$, it achieves its maximum. The assertion in the first inequality of \eqref{e.linearlbnd} is that the maximum value of this function is no more than $\frac{1}{\sqrt{\theta}}.$ Suppose, by contradiction, this were false. Let $y_0\in \overline{\Omega_+}$ be a point at which $\zeta_T$ achieves its maximum, and 
\begin{align}\label{e.conthyp}
    \zeta_T(y_0) > \frac{1}{\sqrt{\theta}}.
\end{align}
There are three possibilities, which we will now argue can never occur: 
\begin{enumerate}
    \item \underline{$y_0 \in \overline{\Omega_+} \cap \partial (TQ_\nu):$}  by virtue of \eqref{e.metriccom}, along $\partial (TQ_\nu)$ we know that $w_T(y) = h_\nu(y) \geqslant \sqrt{\theta}(y \cdot \nu)$ . This implies that $w_T(y) + \eta_0 > \sqrt{\theta}(y\cdot \nu)$ for every $y \in \partial (TQ_\nu) \cap \overline{\Omega_+}.$ Thus, under the contradiction hypothesis \eqref{e.conthyp}, $\zeta_T$ cannot attain its maximum here.
    \item \underline{$y_0 \in \Omega_+:$} in this case, $y_0$ would be an interior maximum point of $\zeta_T,$ and so, 
    \begin{align} \label{e.intmax}
        \nabla \zeta_T(y_0) = 0, \quad \quad \Delta \zeta_T(y_0) \leqslant 0. 
    \end{align}
    Towards ruling out this case, we derive the PDE satisfied by $\zeta_T.$ From the definition of $\zeta_T,$ we note that at any $y \in \Omega_+,$
    \begin{align} \label{e.gradzeta}
        \nu = (w_T(y) + \eta_0) \nabla \zeta_T(y) + \zeta_T(y) \nabla w_T(y).
    \end{align}
Taking divergence of this relation and applying \eqref{e.wtpde}, we find that at any $y \in \Omega_+$,
\begin{equation}   \label{e.zetapde} \begin{aligned}
        &0 = 2 \nabla \zeta_T(y) \cdot \nabla w_T(y) + (w_T(y) + \eta_0)\Delta \zeta_T(y) + \zeta_T(y) \Delta w_T(y)\\
        &\quad = 2 \nabla \zeta_T(y) \cdot \nabla w_T(y) + (w_T(y)+ \eta_0) \Delta \zeta_T(y) \\ &\quad \quad + \frac{4}{\sqrt{2}} \zeta_T(y) \tanh (\sqrt{2}w_T(y)) \left( |\nabla w_T(y)|^2 - a(y) \right).
    \end{aligned}\end{equation}
     In order to evaluate \eqref{e.zetapde} at $y=y_{0}$, we note that from \eqref{e.gradzeta} and \eqref{e.intmax}, we have
    \begin{align*}
        \nu = \zeta_T(y_0) \nabla w_T(y_0).
    \end{align*}
By the contradiction hypothesis \eqref{e.conthyp}, this yields
    \begin{align} \label{e.strictgrt}
        |\nabla w_T(y_0)| = \frac{1}{\zeta_T(y_0)} < \sqrt{\theta}.
    \end{align}
    Moreover, the contradiction hypothesis \eqref{e.conthyp} also guarantees that $y_{0}\cdot \nu>0$, since $y_{0}\in \Om_{+}$. Inserting this into \eqref{e.zetapde} at $y = y_0$, and applying \eqref{e.intmax}, \eqref{e.strictgrt}, and $a\geq \theta$, we have
    \begin{align*}
        0 &= \Delta \zeta_T(y_0) + \frac{4}{\sqrt{2}} \zeta_T(y_0) \frac{\tanh (\sqrt{2}w_T(y_0))}{w_T(y_0) + \eta_0} \left( \frac{1}{\zeta_T^2(y_0)} - a(y_{0}) \right)\\
        &< \frac{4}{\sqrt{2}}(y_{0}\cdot \nu)\tanh (\sqrt{2}w_{T}(y_{0}))(\theta-\theta),
    \end{align*}
  which yields a contradiction. 
    \item $y_0 \in \Omega_+ \cap \{w_T = 0\}:$ finally, if this were to hold, we would have $w_T(y_0) = 0,$ so that 
    \begin{align*}
        \zeta_T(y_0) = \frac{y_0 \cdot \nu}{\eta_0} > \frac{1}{\sqrt{\theta}},
    \end{align*}
    i.e., $y_0 \cdot \nu > \frac{\eta_0}{\sqrt{\theta}} = \eta(0,\nu).$ But $w_T(y_0) = 0$ implies that $u_T(y_0) = \tanh(w_T(y_0))$ is 0, and by Proposition \ref{l.cc} we must have $|y_0 \cdot \nu| \leqslant \eta(0,\nu)$, provided $T \gg 1.$ We again conclude in a contradiction. 
\end{enumerate}
This implies that the contradiction hypothesis \eqref{e.conthyp} cannot hold, and the proof of the lower bound in the first equation in \eqref{e.linearlbnd} is complete. The proof of the other inequalities is similar, with only minor differences.
\end{proof}
Having proven Proposition \ref{p.wt}, we are able to get fine exponential decay estimates for the function $u_T$ and its derivatives, away from the interface $\Sigma_\nu. $ 
\begin{proposition} \label{p.gradestut}
For all fixed $T$ sufficiently large, 
\begin{align} \label{e.potut}
    1 - u^2_T(y) \leqslant C e^{-c|y \cdot \nu| }, \quad \quad y \in TQ_\nu.
\end{align}
Here, $c = 4 \sqrt{\theta}$.
Moreover, there exists a universal constant $C_1 > 0$ such that for all $T \gg 1,$
\begin{align} \label{e.graduT}
    |\nabla u_T(y)| \leqslant C_1e^{-c|y \cdot \nu|}
\end{align} 
\end{proposition}
\begin{proof}
The first inequality is immediate by noting that $1 - u_T^2 = 1 - \tanh^2 (\sqrt{2} w_T) = \sech^2 (\sqrt{2} w_T),$ and $w_T$ satisfies the estimates in Proposition \ref{p.wt}, and \eqref{e.tanhsech}. For the second, by the Euler-Lagrange equations we know that 
$$
|\Delta u_T(y)| = |a(y) W^\prime(u_T)|=|4a(y)u_{T}(1-u_{T}^{2})| \leqslant C e^{-c |y \cdot \nu|} \quad \quad y \in TQ_\nu. 
$$
Rescaling, by setting $y = Tx$ and defining $v_T(x) := u_T(Tx),$ we find that 
\begin{align*}
    |\Delta v_T(x)| = T^2 a(Tx)|W^\prime(v_T(x))| \leqslant CT^2 e^{-c T|x \cdot \nu|}, \quad \quad x \in Q_\nu
\end{align*}
Elliptic estimates yield 
\begin{align*}
    |\nabla v_T(x)| \leqslant C_1T e^{-cT|x \cdot \nu|}. 
\end{align*}
Scaling back, one recovers \eqref{e.graduT}. 
\end{proof}
\section{Bounds on the Error Term} \label{s.errorbound}
Recall the remainder term $\lambda_0$ introduced in \eqref{e.lambda0}. The main result of this section is next. 
\begin{proposition} There exists a constant $\Lambda_0 > 0$ such that
\begin{align} \label{e.errortermfin}
\la_{0}(\nu) \leqslant \La_{0} \quad \mbox{ for all } \nu \in \mathbb{S}^{N-1}.
\end{align}
\end{proposition}
\begin{proof}
We know that $|\nabla h_\nu(y)|\leqslant \sqrt{\Theta}.$ Moreover, from Proposition \ref{p.wt} and \ref{p.gradestut} we have that
\begin{align*}
    |\nabla \phi(u_T(y))| = |\phi^\prime(u_T(y)) \nabla u_T(y)| = \sqrt{2}(1 - u_T^2(y)) |\nabla u_T (y)| \leqslant C e^{-c|y \cdot \nu|}, 
\end{align*}
and, similarly,
\begin{align*}
    |\nabla \phi(q \circ h_\nu)| = \sqrt{2}(1 - \tanh^2(\sqrt{2} h_\nu))|\nabla h_\nu| \leqslant Ce^{-c|y \cdot \nu|}. 
\end{align*}
Then, for $\la_{0}(\nu)$ defined by \eqref{e.lambda0},
\begin{align*}
    &\la_{0}(\nu)=\limsup_{T \to \infty}\frac{1}{T^{N-1}} \left| \int_{TQ_\nu} \nabla (\phi(u_T(y)) - \phi(q \circ h_\nu)) \cdot \nabla h_\nu\right| \\ &\quad 
    \leqslant \sqrt{\Theta} \limsup_{T \to \infty} \frac{1}{T^{N-1}}\int_{TQ_\nu} C e^{-c|y \cdot \nu|}\,dy \\
    &\leqslant C\sqrt{\Theta} \int_{-\infty}^\infty e^{-c|s|}\,ds = \frac{C\sqrt{\Theta}}{\sqrt{\theta}e}=:\La_{0},
\end{align*}
where, we recall from Proposition \ref{p.gradestut} that $c = 4\sqrt{\theta}. $
\end{proof}
\begin{proof}[Proof of Theorem \ref{t.sigmarep}]
As discussed in the introduction, the proof of Theorem \ref{t.sigmarep} is an immediate consequence of \eqref{e.errortermfin}. 
\end{proof}

\section{The Proof of Theorem \ref{t.planarint} and Theorem \ref{t.ap}} \label{s.pmp}
We begin by summarizing several properties of $h_{\nu}$ that will be needed in the proof of Theorem \ref{t.planarint}. 
\begin{lemma}
\label{l.per}
Let $\nu \in \mathbb{S}^{N-1} \cap \mathbb{Q}^N.$ There exist $T_0 \in \mathbb{N}$ and unit vectors $\{\nu_i\}_{i=1}^{N-1}\subseteq \mathbb{S}^{N-1} \cap \mathbb{Q}^N$ such that $\{\nu_1,\cdots, \nu_{N-1}, \nu_N := \nu\}$ form an orthonormal basis for $\RR^N.$ Moreover, the coefficient $a$ is $T_0$ periodic in the directions $\{\nu_i\}_{i=1}^N$, and $h_\nu$ is $T_0$ periodic in the directions $\{\nu_i\}_{i=1}^{N-1}.$ 
\end{lemma}
\begin{proof}
By an appeal to \cite[Proposition 3.5]{CFHP}, there exist $\nu_1,\cdots,\nu_{N-1} \in \mathbb{Q}^N\cap \mathbb{S}^{N-1}$ and $T_0 \in \mathbb{N}$ such that $\{\nu_i\}_{i=1}^N$ is an orthonormal basis of $\RR^N,$ and $a$ is $T_0-$periodic in each of the directions $\{\nu_i\}_{i=1}^N.$ We prove the periodicity of $h_\nu$ in the directions $\left\{\nu_{i}\right\}_{i=1}^{N-1}$. We fix $x \in \RR^N,$ and show that for any $i \in \{1,\cdots, N-1\},$ 
\begin{align*}
h_\nu(x + kT_0 \nu_i) = h_\nu(x), \mbox{ for all } k \in \mathbb{Z}.
\end{align*}
We note that if $x\cdot \nu=0$, then the estimate is automatic since both sides of the equation are 0. Without loss of generality, we may assume that $x \cdot \nu > 0$ and $k \geqslant 0.$ Let $y \in \Sigma_\nu$ and $\gamma:[0,1] \to \RR^N$ be such that $\gamma(0) = x, \gamma(1) = y ,$ and $\int_0^1 \sqrt{a(\gamma(t))}|\dot{\gamma}(t)| \,dt = h_\nu(x).$ The existence of such a geodesic follows by classical arguments. For each $\nu_{i}$, $i=1, \ldots, N-1$, we define $\tilde{\gamma}:[0,1] \to \RR^N$ by $\tilde{\gamma}(t) := \gamma(t) + kT_0 \nu_i.$  Since $\nu_{i}\perp\nu$, we have $\tilde{\gamma}(1)\cdot \nu=\ga(1)\cdot \nu+kT_{0}\nu_{i}\cdot \nu=y\cdot \nu=0$, which implies  $\tilde{\ga}(1) \in \Sigma_\nu$. We also note that $\tilde{\gamma}(0) = x + k T_0 \nu_i$. Hence, by the $T_{0}$ periodicity of $a$ with respect to $\nu_{i},$ we have
\begin{align*}
&h_\nu(x + kT_0 \nu_i) = d_{\sqrt{a}}(x + kT_0 \nu_i,\Sigma_\nu) \leqslant d_{\sqrt{a}} (x + kT_0 \nu_i, y + k T_0 \nu_i)\\
&\quad \quad \leqslant \int_0^1 \sqrt{a(\tilde{\gamma}(t))} |\dot{\tilde\gamma}(t)|\,dt = \int_0^1 \sqrt{a(\gamma(t) + kT_0 \nu_i)} |\dot{\gamma}(t)|\,dt \\
&\quad \quad =\int_0^1 \sqrt{a(\gamma(t))}|\dot{\gamma}(t)|\,dt= d_{\sqrt{a}}(x,y) = h_\nu(x).
\end{align*}
The reverse inequality follows by a symmetric argument.
\end{proof}
We now make a slight digression to almost periodic functions, which will play an important role in the characterization of the asymptotic behaviour of $h_{\nu}$ (see Lemma \ref{l.hompmp}). When $\nu \in \mathbb{S}^{N-1} \cap \mathbb{Q}^N,$ we know from Lemma \ref{l.per} that there is an orthonormal basis $\{\nu_1,\cdots, \nu_N := \nu\}\subseteq\mathbb{S}^{N-1}\cap \mathbb{Q}^{N}$, and $T_0 = T_0(\nu) \in \mathbb{N},$ such that $h_\nu$ is $T_0-$periodic in the transverse directions $\{\nu_i\}_{i=1}^{N-1}.$ {This periodicity yields an averaging property which we will exploit in the proof of Lemma \ref{l.hompmp}. When $\nu \in \mathbb{S}^{N-1} \backslash \mathbb{Q}^N,$ it turns out that an averaging property still holds, using the theory of Bohr almost periodic functions}. For the convenience of the reader, we recall the basic notions of the theory of Bohr almost periodic functions, referring to \cite{Besicovitch} for details. 
\begin{definition}\label{d.ap1}
A continuous bounded function $g: \RR^d \to \RR$ is said to be \textit{Bohr almost periodic} if for every $\eta > 0,$ there exists an {\it $\eta-$almost period} $\tau > 0$ such that for any $\alpha \in \RR^d,$ there exists $\zeta \in \alpha + \tau \blacksquare_d$ with 
\begin{align}\label{e.bap}
\sup_{x \in \RR^d} |g(x + \zeta) - g(x) | \leqslant \eta,
\end{align}
where $\blacksquare_d$ is any $d$-dimensional unit cube. 
\end{definition}

\begin{remark}
In the sequel, we use almost periodicity primarily with $d = N-1.$ Continuous periodic functions are examples of Bohr almost periodic functions (by choosing $\tau$ larger than the period, since then \eqref{e.bap} holds with $\eta=0$). 

An important feature of Bohr almost periodic functions, which we will use in the proof of Lemma \ref{l.hompmp}, is the existence of the so-called mean value. To be precise, if $f$ is a Bohr almost periodic function, then the limit 
\begin{align} \label{e.meanvalue}
\mu(f) := \lim_{T \to \infty} \frac{1}{T^d} \int_{T\blacksquare_d} f(y)\,dy = \lim_{T \to \infty} \int_{\blacksquare_d} f(Ty)\,dy
\end{align}
exists and is finite. 
\end{remark}
\begin{remark}
In what follows, We will use the definition of Bohr almost periodicity with various choices of the unit cube $\blacksquare_d$, as it turns out that the definition, and the mean value defined above, are independent of the choice of the unit cube $\blacksquare_d.$ To be precise,  let $\{V_k\}_{k=1}^\infty \subseteq \RR^N$  be a sequence of bounded domains with $\mathcal{L}^d(V_k) \to \infty$ as $k \to \infty,$ and let $\{V_k^{h}\}$ denote the set of points in $V_k$ at distance not exceeding $h$ from the boundary $\partial V_k$. If $\partial V_{k}$ is regular enough such that there exists a sequence $(h_k)_{k=1}^\infty$ with $h_k \to 0$ and $\lim_{k \to \infty} \frac{\mathcal{L}^d(V_{k}^{h_k})}{\mathcal{L}^d(V_k)} = 0,$ then the limit in \eqref{e.meanvalue} is equal to 
\begin{align*}
    \mu(f) = \lim_{k \to \infty} \fint_{V_k} f(y)\,dy. 
\end{align*}
For a proof of this assertion, see \cite[Proposition 1.9]{subin}.
\end{remark}
It is well known that $f$ is Bohr almost periodic if and only if $f$ has a uniformly convergent Bochner-Fourier series (see \cite{Besicovitch}). In particular, if $f$ is Bohr almost-periodic, then there exist an at most countable set $\Lambda \subseteq \RR^d$ of ``frequencies'', and a square-summable sequence $\{f_\lambda\}_{\lambda \in \Lambda} \subseteq \mathbb{C}$ of ``Fourier modes'', such that 
\begin{align}\label{e.fourier}
f(x) = \sum_{\lambda \in \Lambda} f_\lambda e^{i \lambda \cdot x} \quad\text{for $x \in \RR^d$,}
\end{align}
and the sum on the right is uniform and absolute. The coefficients $f_\lambda$ are given by $f_\lambda := \mu(fe^{-i\lambda \cdot (\cdot)})$ for $\mu$ as in \eqref{e.meanvalue}, and $\Lambda \subseteq \RR^d$ is the at most countable set for which $f_\lambda \neq 0.$  In particular, this implies that if $f$ is Bohr almost periodic, and  
\begin{equation}\label{l.fsnought}
\mu (f(\cdot) e^{-i\lambda \cdot (\cdot)}) = 0\quad\text{for every $\lambda \in \RR^d$},
\end{equation}
then $f \equiv 0.$ 

We will also use the notion of two-scale convegence for Bohr almost periodic functions \cite[Definition 4.1, Proposition 4.6]{CDGay}. We introduce the space $\mathcal{B}^1$ as the closure of Bohr almost periodic functions with respect to the semi-norm 
\begin{align*}
[f] := \lim_{T \to \infty} \frac{1}{T^d}\int_{T\blacksquare_d} |f(y)|\,dy = \mu(|f|).
\end{align*}
\begin{definition} \label{d.2sc}
Let $\Omega \subseteq \RR^d$ be open. We say that a sequence $\{u_\eta\} \subseteq L^1_{\mathrm{loc}}(\Omega)$ Bohr two-scale converges to $u \in L^1_{\mathrm{loc}}(\Omega;\mathcal{B}^1)$ if for every bounded function $g:\Omega \times \RR^d \to \RR$ that is continuous in the first variable and Bohr almost periodic in the second variable, we have 
\begin{align*}
\lim_{\eta \to 0} \int_\Omega u_\eta (x) g\left(x,\frac{x}{\eta}\right)\,dx = \int_\Omega \mu(u(x,\cdot)g(x,\cdot))\,dx.
\end{align*}
\end{definition}
\begin{remark}\label{r.mean}
It is proven in \cite[Proposition 4.6]{CDGay} that if $f$ is a Bohr almost periodic function, and $T_m \to \infty$ is a sequence of positive numbers, then $f_m (\cdot) := f(T_m \, \cdot)$ Bohr two-scale converge to $\mu(f)$ in any bounded open set $\Omega \subseteq \RR^d.$ 
\end{remark}
We next show that for each $\nu\in \mathbb{S}^{N-1}$, the function  $x \mapsto \frac{h_{\nu}(x)}{x\cdot \nu}$ satisfy Bohr almost periodicity as functions of the orthogonal directions. 
\begin{lemma} \label{l.almper}
Let $\nu \in \mathbb{S}^{N-1}$, and write $x \in \RR^N$ as $x = x^\prime + (x \cdot \nu)\nu \sim (x^\prime,x \cdot \nu).$ The functions $x^\prime \in \Sigma_\nu  \mapsto a(x^\prime, s)$ and $x^\prime \in \Sigma_\nu \mapsto \frac{h_\nu(x^\prime, s)}{s}$ are Bohr-almost periodic for every $s \in \RR\setminus \left\{0\right\}$, uniformly in $s.$ To be precise, for every $\eta > 0$ there exists $\tau > 0$, independent of $s$, such that for any $\alpha \in \Sigma_\nu$, there exists $\zeta \in \alpha + (\tau Q_\nu \cap \Sigma_\nu)$ such that 
\begin{align}\label{e.1star}
\sup_{x'\in \Sigma_\nu}|a(x^\prime + \zeta, s) - a(x^\prime,s)| \leqslant \eta, 
\end{align}
and 
\begin{align}\label{e.2star}
\sup_{x'\in \Sigma_\nu}\left| \frac{h_\nu(x^\prime + \zeta, s)}{s} - \frac{h_\nu(x^\prime, s)}{s}\right| \leqslant \sqrt{\frac{\Theta}{\theta}}\eta. 
\end{align}  
\end{lemma}
\begin{proof}
We recall $\square_{\nu}:=Q_{\nu}\cap \Sigma_{\nu}$, and throughout the proof of the Lemma we use this choice of an $(N-1)-$dimensional unit cube $\blacksquare_{N-1}$ from the definition of Bohr-almost periodicity. As $a$ is $\mathbb{T}^N-$periodic and smooth, it admits an absolutely and uniformly convergent Fourier series 
\begin{align*}
a(x) = \sum_{k \in \mathbb{Z}^N} a_k e^{2\pi i k \cdot x}, \quad \quad x \in \RR^N.
\end{align*}
Upon a rotation, we may express $x=(x', x\cdot \nu)$ and $k = (k^\prime, k \cdot \nu)$, and take the sum over another countable family $\La^{N}$ which is isomorphic to $\mathbb{Z}^{N}$. We let $x_{N}=x \cdot \nu = s \in \RR\setminus \left\{0\right\}$ be fixed. Defining $b_k := a_k e^{2\pi i (k \cdot \nu) s}=a_{k}e^{2\pi i k_{N} x_{N}},$ we find that $|b_k| = |a_k|$, and we have
\begin{align*}
a(x^\prime, s) = \sum_{k' \in \La^{N-1}}\Big(\sum_{\substack{k\in \La^{N}:\\ k=(k', \cdot)}} b_k\Big)e^{2\pi i k^\prime \cdot x^\prime}, \quad \quad x^\prime \in \Sigma_\nu.
\end{align*}
Since the series on the right converges uniformly and absolutely, it follows that $a(\cdot,s)$ is Bohr almost periodic. In particular, for each $\eta > 0$ there exists $\tau > 0$ such that for every $\alpha \in \Sigma_\nu,$ there exists $\zeta \in \alpha + \tau\square_{\nu}$ satisfying 
\begin{align*}
\sup_{x^\prime \in \Sigma_\nu} |a(x^\prime + \zeta,s) - a(x^\prime,s)| \leqslant \eta,
\end{align*} 
and this proves \eqref{e.1star}. The property of almost periodicity is preserved under composition with uniformly continuous functions. As a consequence, for each $\eta > 0$ there exists $\tau > 0$ such that for all $\alpha \in \Sigma_\nu$, there exists $\zeta \in \al+\tau \square_{\nu}$ with
\begin{align}\label{e.sqrtaap}
\sup_{x^\prime \in \Sigma_\nu} |\sqrt{a(x^\prime + \zeta,s)} - \sqrt{a(x^\prime,s)}| \leqslant \eta.
\end{align} 
The proof of almost periodicity of $\frac{h_\nu(\cdot, s)}{s}$ follows a similar argument as the proof of Lemma \ref{l.per}. Fix $x \in \RR^N$ and, without loss of generality, assume that $x \cdot \nu > 0.$ Let $y \in \Sigma_\nu$ and $\gamma:[0,1] \to \RR^N$ be such that $\gamma(0) = x, \gamma(1) = y$ and $\int_0^1 \sqrt{a(\gamma(t))}|\dot{\gamma}(t)| \,dt = h_\nu(x).$ Let $\zeta\in \Sigma_{\nu}$, and we define $\tilde{\gamma}:[0,1] \to \RR^N$ by $\tilde{\gamma}(t) := \gamma(t) + \zeta.$ Note that as $\zeta \perp \nu$, we have $\tilde{\gamma}(1) \cdot \nu = \gamma(1) \cdot \nu + \zeta \cdot \nu = y \cdot \nu = 0,$ so that $\tilde{\gamma}(1) \in \Sigma_\nu.$ Moreover, $\tilde{\gamma}(0) = x + \zeta,$ and so
\begin{align*}
h_\nu(x + \zeta) &\leq d_{\sqrt{a}} (x+\zeta,y) \\
&\leqslant \int_0^1 \sqrt{a(\tilde\gamma(t))}|\dot{\tilde\gamma}(t)|\,dt\\
&= \int_0^1 \sqrt{a(\gamma(t))}|\dot{\gamma}(t)|\,dt + \int_0^1 \big( \sqrt{a(\gamma(t) + \zeta)} - \sqrt{a(\gamma(t))}\big)|\dot{\gamma}(t)|\,dt\\
&= h_{\nu}(x) + \int_0^1 \big( \sqrt{a(\gamma(t) + \zeta)} - \sqrt{a(\gamma(t))}\big)|\dot{\gamma}(t)|\,dt.
\end{align*}
Choose $\zeta\in \Sigma_{\nu}$ as in \eqref{e.sqrtaap}, and conclude that
\begin{align}\label{e.3star}
|h_\nu(x+\zeta) - h_\nu(x)| \leqslant \eta \int_0^1 |\dot{\gamma}(t)|\,dt &\leqslant \frac{\eta}{\sqrt{\theta}} \int_0^1 \sqrt{a(\ga(t))}|\dot{\gamma}(t)|\,dt\notag \\
&\leqslant \frac{\eta}{\sqrt{\theta}}\sqrt{\Theta}|x\cdot\nu|, 
\end{align}
{where in the last inequality we have used the definition of $\ga(t)$ and its relation to $h_{\nu}(x)$, as well as \eqref{e.metriccom}.} The inequality \eqref{e.2star} now follows upon diving \eqref{e.3star} through by $|x \cdot \nu|$, and noting that $\zeta \cdot \nu = 0.$

\end{proof}

The next lemma is crucial for the proof of Theorem \ref{t.planarint}, and requires various properties of $h_{\nu}$ which we have previously established.   
\begin{lemma} \label{l.hompmp}
Fix $\nu \in \mathbb{S}^{N-1},$ and let $\{T_m\}_{m\in \mathbb{N}} \subseteq (0,\infty)$ with $T_m\to \infty$ as $m \to \infty.$ For $m \in \mathbb{N},$ consider the functions $k_{m}:\mathbb{R}^{N}\rightarrow \RR$ defined as
\begin{equation}\label{e.kmdef}
k_m(\cdot) := \frac{1}{T_m} h_\nu(T_m \,\cdot).
\end{equation}
There exists a constant $c(\nu) \in [\sqrt{\theta},\sqrt{\Theta}]$ and a subsequence of $\{T_m\}_{m\in \mathbb{N}}$ (which we do not relabel) such that for any compact set $K\subseteq \RR^{N}\setminus \Sigma_{\nu}$, and for every $\alpha > 0,$ there exists $M=M(\alpha, K) \in \mathbb{N}$ such that if $m \geqslant M$, then 
\begin{align} \label{e.convdstf}
\big| k_m (z) - c(\nu) z \cdot \nu\big| \leqslant \alpha |z \cdot \nu| \quad \text{for all {$z\in K$}.}
\end{align} 
\end{lemma}
\begin{proof}
We show that $\left\{k_{m}\right\}_{m\in \mathbb{N}}$ is uniformly bounded and uniformly Lipschitz, from which we obtain local uniform convergence (up to a subsequence) in a strong (uniform) topology. We further use averaging associated to weak convergence arguments to identify the limit in a weak topology. Carrying out this program involves some ideas using polynomial approximation which might be of independent interest in this context. We break up the proof in several steps. 

\par \textbf{Step A:} We show that there exists a Lipschitz continuous function $k: \mathbb{R}^{N}\rightarrow \RR$ and a subsequence of $\left\{k_{m}\right\}$ (which we do not relabel) such that for any compact $K\subseteq \RR^{N}\setminus \Sigma_{\nu}$ and for every $\alpha > 0$, there exists $M=M(\alpha, K) \in \mathbb{N}$ such that, if $m \geqslant M$,
\begin{align} \label{e.ucon}
|k_m(z) - k(z)| \leqslant \alpha |z \cdot \nu|\quad\text{for all $z\in K$.}
\end{align}
As $\nu$ is fixed, we define $z_{N}:=z\cdot \nu$ and write $z=(z', z_{N})=(z', z\cdot\nu)$ throughout the rest of the proof.  
By \eqref{e.metriccom}, for $z\in \mathbb{R}^{N}\setminus \Sigma_{\nu}$ we have
\begin{equation}\label{e.kmbnd}
\left|\frac{k_{m}(z)}{z_{N}}\right|=\left|\frac{1}{T_{m}z_{N}}h_{\nu}(T_{m}z)\right|\leq \left|\frac{\sqrt{\Theta}T_{m}z_{N}}{T_{m}z_{N}}\right|= \sqrt{\Theta}.
\end{equation}
By Lemma \ref{l.hlips} and \eqref{h.nondeg}, $k_m$ is Lipschitz with 
\begin{align}\label{e.kmgrad}
\|\nabla k_m\|_{L^\infty} =\norm{\nabla h_{\nu}}_{L^{\infty}}\leqslant \sqrt{\Theta}. 
\end{align}
Combining \eqref{e.kmbnd} and \eqref{e.kmgrad}, we deduce that for a point of differentiability $z\in \RR^{N}\setminus \Sigma_{\nu}$, 
\begin{align}\label{e.aagrad}
\left|\nabla \left(\frac{k_{m}(z)}{z_{N}}\right)\right|&=\left|\nabla \left(\frac{k_{m}(z)}{z\cdot\nu}\right)\right|\notag\\
&= \left|\frac{\nabla k_{m}(z)(z\cdot\nu)-k_{m}(z)\nu}{(z\cdot\nu)^{2}}\right|\notag\\
&\leq \frac{2\sqrt{\Theta}}{|z\cdot\nu|}=\frac{2\sqrt{\Theta}}{|z_{N}|}.
\end{align}

In view of \eqref{e.kmbnd} and \eqref{e.aagrad}, the Arzel\`{a}-Ascoli theorem yields that there exist a subsequence of $\left\{k_{m}\right\}$ (not relabeled) and a continuous function $\tilde{q}: \RR^{N}\setminus \Sigma_{\nu}\rightarrow \RR$ such that, for every compact set $K\subseteq \RR^{N}\setminus \Sigma_{\nu}$, 
\begin{equation}\label{e.aabase}
\lim_{m\rightarrow \infty} \sup_{z\in K} \left|\frac{k_{m}(z)}{z_{N}}-\tilde{q}(z)\right|=0. 
\end{equation}
Defining now 
\begin{equation*}
k(z):=\begin{cases}\tilde{q}(z)z_{N}&\text{for $z\in\RR^{N}\setminus \Sigma_{\nu},$}\\
0&\text{for $z\in \Sigma_{\nu}$,}
\end{cases}
\end{equation*}
we see that \eqref{e.ucon} follows from \eqref{e.aabase}.

\textbf{Step B:} Fix $R>1$. We argue that $h_{\nu}$ can be approximated on $\RR^{N-1}\times [-R,R]$ by a polynomial in the last variable. In particular, we will show that this polynomial belongs to the class
\begin{equation*}
\mathfrak{A}:=\left\{g(z', z_{N}):=\sum_{j=0}^{p}b_{j}(z')z_{N}^{j}: p\in \mathbb{N}, b_{j}\in AP(\RR^{N-1})\right\},
\end{equation*}
where $AP(\RR^{N-1})$ is the set of Bohr almost periodic functions in $\RR^{N-1}$. 
In what follows, we write $z \in \RR^N$ as $z = (z^\prime, z_N)$ with $z^\prime \in \Sigma_\nu \sim \RR^{N-1},$ and $|z_N| \leqslant R.$ Let $f:\RR^{N-1}\times [0,1] \to \RR$ be defined by 
\begin{align*}
f(z^\prime,z_N) := h_\nu\left(z^\prime, 2Rz_N-R\right). 
\end{align*}
Fix $z'\in \RR^{N-1}$ such that $(z^\prime, 0) \in \Sigma_\nu,$ and consider the functions $z_{N}\rightarrow f(z', z_{N})$. Throughout the rest of Step B, we allow $C=C(N, \theta, \Theta)$ in every step.
Define $\tilde{f}: \RR^{N-1}\times[0,1] \to \RR$, by 
\begin{align}\label{e.f*def}
\tilde{f}(z', z_{N}):&=f(z', z_{N})-f(z',0)-z_{N}\left(f(z', 1)-f(z',0)\right)\notag\\
&=h_{\nu}(z', 2Rz_{N}-R)-h_{\nu}(z', -R)-z_{N}(h_{\nu}(z',R)-h_{\nu}(z', -R)). 
\end{align}
We have that $\tilde{f}(z',0)=\tilde{f}(z', 1)=0$, and we can extend this function to all of $\RR^{N}$ by setting $\tilde{f}=0$ off of $\RR^{N-1}\times [0,1].$

We now proceed nearly identically to the proof of the Stone-Weierstrass Theorem (see  \cite[Page 160]{Rudin}). For $j \in \mathbb{N}_0$, we define the functions $g_j : \RR^N \to \RR$ as 
\begin{equation}\label{e.gjformula}
g_{j}(z', z_{N}):=C_{j}\int_{-1}^{1}\tilde{f}(z', z_{N}+t)(1-t^{2})^{j}\, dt,
\end{equation}
where $C_{j}$ are chosen so that 
\begin{equation}\label{e.int1}
\int_{-1}^{1}C_{j}(1-t^{2})^{j}\, dt=1, 
\end{equation}
and, as shown in \cite{Rudin},
\begin{equation}\label{e.cjbnd}
|C_{j}|\leq C\sqrt{j}.
\end{equation}
Since $\tilde{f}\equiv 0 $ in $\RR^{N-1} \times (\RR \backslash [0,1])$, we have that for any $z'\in \RR^{N-1}$ and $z_{N}\in [0,1]$, 
\begin{equation*}
g_{j}(z', z_{N})=C_{j}\int_{-z_{N}}^{1-z_{N}}\tilde{f}(z', z_{N}+t)(1-t^{2})^{j}\, dt=C_{j}\int_{0}^{1}\tilde{f}(z', t)(1-(t-z_{N})^{2})^{j}\, dt, 
\end{equation*}
which is a polynomial in $z_{N}$ with continuous coefficients depending on $z'$. Recall that by Lemma \ref{l.almper}, $\frac{h_{\nu}(\cdot, s)}{s}$ is Bohr almost periodic, uniformly, for all $s\neq 0$. In particular, this implies that $h_{\nu}(\cdot, s)$ is Bohr almost periodic (the case $s=0$ being trivial since $h_{\nu}(\cdot, 0)=0$). Note that for every $z_{N}\in [0,1]$, $\tilde{f}(\cdot, z_{N})$ defined by \eqref{e.f*def} is a linear combination of Bohr almost periodic functions, which is still Bohr almost periodic. We infer that $g_{j}(\cdot, z_{N})$, whose coefficients are given by integration in the $N$th variable of Bohr almost periodic functions (which does not affect the first $N-1$ variables), is Bohr almost periodic for every $z_N \in [0,1].$ 

%

For $z\in \RR^{N-1}\times [0,1]$, we define 
\begin{equation*}
M(z)=M(z',z_{N}):=\max_{t\in [-1,1]} |\tilde{f}(z', z_{N}+t)|.
\end{equation*}

By \eqref{e.f*def} and the Lipschitz continuity of $h_{\nu}$ (Lemma \ref{l.hlips} and \eqref{h.nondeg}), we have
\begin{align}\label{e.mbound}
M(z)&= \max_{t\in [-1,1]}\large|h_{\nu}(z', 2R(z_N+t)-R)-h_{\nu}(z', R)\notag\\
&\quad -(z_{N}+t)(h_{\nu}(z', R)-h_{\nu}(z', -R))\large|\notag\\
&\leq  \max_{t\in [-1,1]}\sqrt{\Theta}\left|2R(z_{N}+t-1)\right|+\left|z_{N}+t\right| 2\sqrt{\Theta}R\notag\\
&= CR\left(z_{N}+1\right).
\end{align}

Fix $\eta>0$. Note that by Lemma \ref{l.hlips}, $f$ (and hence $\tilde{f}$) is Lipschitz continuous and, in particular, $f$ (and hence $\tilde{f}$) is uniformly continuous. Hence, choose $\delta\in(0,1)$ such that for any $x,y\in \RR^N$ with $|x-y|\leq \delta$, we have $|\tilde{f}(x)-\tilde{f}(y)|\leq \frac{\eta}{2}$.

%
%
By \eqref{e.gjformula}, \eqref{e.int1}, the uniform continuity of $\tilde{f}$, \eqref{e.mbound}, and \eqref{e.cjbnd}, we find that at $z\in \RR^{N-1}\times [0,1]$, 
\begin{align*}
|\tilde{f}(z) - g_{j}(z)|&=\left|C_{j}\int_{-1}^{1}\left(\tilde{f}(z', z_{N})-\tilde{f}(z', z_{N}+t)\right)(1-t^{2})^{j}\, dt\right|\\
&\leq C_{j}\int_{(-1,1)\setminus (-\delta, \delta)}\left|\tilde{f}(z', z_{N}+t)-\tilde{f}(z', z_{N})\right|(1-t^{2})^{j}\, dt\\
&\quad +C_{j}\int_{-\delta}^{\delta}\left|\tilde{f}(z', z_{N}+t)-\tilde{f}(z', z_{N})\right|\left(1-t^{2}\right)^{j}\, dt\\
&\leq  C_{j}\int_{(-1,1)\setminus (-\delta, \delta)}2M(z)(1-t^{2})^{j}\, dt+C_{j}\int_{-\delta}^{\delta}\frac{\eta}{2}(1-t^{2})^{j}\, dt\\
&\leq CM(z)\sqrt{j}(1-\delta^{2})^{j}+\frac{\eta}{2}\\
&\leq CR\sqrt{j}(1-\delta^{2})^{j}\left(z_{N}+1\right)+\frac{\eta}{2}.
\end{align*}
Taking $j$ sufficiently large, and using the fact that $R>1$, and that $$\lim_{j \to \infty} \sqrt{j}(1 - \delta^2)^j = 0,$$ we can find $g^{\eta}\in \mathfrak{A}$ such that, for all $z\in \RR^{N-1}\times [0,1]$, 
\begin{equation*}
    |\tilde{f}(z) - g^{\eta}(z)|\leq R\eta(z_{N}+1).
\end{equation*}
By \eqref{e.f*def}, this implies that for all $z\in \RR^{N-1}\times [0,1]$,
\begin{multline*}
\left|h_{\nu}(z', 2Rz_{N}-R)-h_{\nu}(z', -R)-z_{N}(h_{\nu}(z',R)-h_{\nu}(z', -R))-g^{\eta}(z)\right|\\
\leq R\eta(z_{N}+1). 
\end{multline*}
Combining the constant term and the linear term in $z_{N}$ into the polynomial $g^{\eta}(z),$ we deduce that there is a polynomial $g^{\eta,R}\in \mathfrak{A}$ such that 
\begin{equation*}
\left|h_{\nu}(z', 2Rz_{N}-R)-g^{\eta, R}(z)\right|\leq R\eta(z_{N}+1)\quad \text{for all $z\in \RR^{N-1}\times [0,1]$.}
\end{equation*}
By the affine transformation in the $N$th variable $(z_{N}\mapsto \frac{1}{2R}(z_{N}+R))$, we obtain a polynomial $\tilde{g}^{\eta,R}\in \mathfrak{A}$ such that
\begin{equation}\label{e.polyapprox}
\left|h_{\nu}\left(z\right)-\tilde{g}^{\eta, R}(z)\right|\leq \frac{\eta}{2}(z_{N}+3R) \quad\text{for all $z\in \RR^{N-1}\times [-R,R]$.}
\end{equation}


\textbf{Step C.} In this step, we argue that the linear growth of $h_\nu$ at infinity implies that we may restrict to polynomial approximations that are linear in $z_N.$ Our strategy to make this reduction will be to obtain a single ``infinite polynomial'' which pointwise approximates the bounded function $z \in \RR^N \backslash \Sigma_\nu \mapsto \frac{h_\nu(z)}{z_N}$ (rather than on sets of the form $\RR^{N-1} \times [-R,R],$ in which the coefficients of the polynomial approximation might depend on $R$). 
\par To this end, let $\{\tau_m\}_m$ denote an increasing sequence of positive numbers with $\tau_m \to \infty$ as $m \to \infty.$ We define 
\begin{equation*}
\mathfrak{A}^{-1}:=\left\{g(z', z_{N}):=\sum_{j=-1}^{p}\tilde{b}_{j}(z')z_{N}^{j}: p\in \mathbb{N}, \tilde{b}_{j}\in AP(\RR^{N-1})\right\}.
\end{equation*}
Fix $\eta > 0.$ On the compact interval $[-\tau_{m}, -\frac{1}{\tau_{m}}]\cup[\frac{1}{\tau_m}, \tau_m],$ let $p_m \in \mathfrak{A}^{-1} $ as in Step B be chosen such that 
\begin{align} \label{e.24*}
\left| \frac{h_\nu(z^\prime, z_N)}{z_N} - p_m(z) \right| \leqslant \frac{\eta}{2^m} \quad \quad z \in \RR^{N-1} \times \left[-\tau_{m}, -\frac{1}{\tau_{m}}\right]\bigcup \left[\frac{1}{\tau_m}, \tau_m\right].
\end{align}
We set $q_0 := p_1, $ and $q_m := p_{m+1} - p_m,$ for $m \in \mathbb{N}.$ Then 
\begin{align*}
p_{m+1}(z) = \sum_{n=0}^m q_n(z),
\end{align*}
and thus, pointwise, we have 
\begin{align} \label{e.singleseries}
\frac{h_\nu(z)}{z_N} = \sum_{n=0}^\infty q_n(z) + O(\eta), \quad \quad z \in \RR^N-1 \times \left(\RR\setminus 0\right).
\end{align}
\par Next, we show that on sets of the form $\RR^{N-1} \times [c,d], 0 < c < d < \infty,$ the series in \eqref{e.singleseries} converges uniformly and absolutely. Indeed, if $[c,d] \subset (0,\infty)$, then for all $m$ sufficiently large and for all $z \in \RR^{N-1} \times [c,d],$ we have by \eqref{e.24*} 
\begin{align*}
\left| \frac{h_\nu(z)}{z_N} - p_m (z) \right| \leqslant \frac{\eta}{2^m},
\end{align*}
and so,
\begin{align*}
|q_m(z)|\leqslant \left|p_{m+1}(z)-\frac{h_\nu(z)}{z_N}\right|+\left|\frac{h_\nu(z)}{z_N}-p_{m}(z)\right|\leqslant \frac{\eta}{2^m} + \frac{\eta}{2^{m+1}}.
\end{align*}
It follows that there exists $M \in \mathbb{N}$ large so that 
\begin{align*}
\sum_{m=M}^\infty |q_m(z)| \leqslant 2 \sum_{m=M}^\infty \frac{\eta}{2^m} \leqslant C \eta,
\end{align*}
and thus the series $\sum_{m=0}^\infty q_m(z)$ converges uniformly and absolutely to $z \mapsto \frac{h_\nu(z)}{z_N}$ on the set $\RR^{N-1} \times [c,d].$ As $\{p_m\}_{m} \subset \mathfrak{A}^{-1},$ we have also $\{q_m\}_m \subset \mathfrak{A}^{-1}.$ Collecting powers of $z_N$ and rearranging, using the absolute summability we may rewrite the series in \eqref{e.singleseries} as 
\begin{align*}
\frac{h_\nu(z)}{z_N} = \sum_{j=-1}^\infty \tilde{b}_j(z^\prime)z_N^j+O(\eta),
\end{align*}
where the coefficients $\tilde{b}_j$ are Bohr almost periodic, and therefore, bounded. Testing this with $z = T_m \zeta,$ for $\zeta \in \overline{Q}_\nu$ with $\zeta_N = 1,$ we get 
\begin{align*}
 \frac{h_\nu(T_m \zeta)}{T_m } = \sum_{j=-1}^\infty \tilde{b}_j(T_m \zeta^\prime) T_m^j  + O(\eta).
\end{align*}
We claim that each of the terms  
\begin{align} \label{e.boundbj}
\sup_m |\tilde{b}_j(T_m \zeta^\prime)T_m^j| \leqslant C_j,
\end{align}
 for some constant $C_j >0,$ for every $j \geqslant 1.$ To see this, as the infinite series above is convergent for each $m,$  there exists $J_0(m)$ such that for all $j \geqslant J_0(m),$ we have $|\tilde{b}_j(T_m \zeta^\prime)T_m^j| \leqslant 1,$ and $\left|\sum_{j \geqslant J_0(m)} \tilde{b}_j(T_m \zeta^\prime)T_m^j\right| \leqslant 1.$ It follows by the triangle inequality that 
\begin{align} \label{e.front}
\sup_m\left| \sum_{j \leqslant J_0(m)} \tilde{b}_j(T_m \zeta^\prime)T_m^j\right| \leqslant C.
\end{align}
As the coefficients $\tilde{b}_j$ are bounded, by Bohr almost periodicity, \eqref{e.front} implies the claim in \eqref{e.boundbj}. 
\par  Having proven the claim in \eqref{e.boundbj}, it follows that the $\tilde{b}_j$ are Bohr almost periodic functions that decay at infinity. From this, we claim that each of the $\tilde{b}_j$ must vanish identically. Indeed, for any $j \geqslant 1,$ since 
$|\tilde{b}_j(T_m \zeta^\prime)| \leqslant \frac{C_j}{T_m^j},$
squaring, and integrating over the cube $\square_\nu,$ and then sending $m \to \infty,$ we find that 
\begin{align*}
\mu(\tilde{b}_j^2) = \lim_{m \to \infty} \int_{\square_\nu} |\tilde{b}_j(T_m \zeta^\prime)|^2 \,d\zeta^\prime  = 0,
\end{align*}
which, together with \eqref{l.fsnought}, implies that $\tilde{b}_j \equiv 0$ for all $j\geq 1$. In particular, we may reduce \eqref{e.polyapprox} to a linear approximation, so that for every $\eta > 0$ and $R > 1$ there exist Bohr almost periodic functions $b^{\eta}_0, b^{\eta}_1$ with 
\begin{align} \label{e.unifapp}
|h_\nu(z) - b_0^\eta(z^\prime) - b_1^\eta(z^\prime) z_N| \leqslant  \eta (|z_N| + R), \quad \quad \mbox{ for all } z \in \RR^{N-1} \times [-R,R].
\end{align}

\par Recalling \eqref{e.kmdef}, by \eqref{e.unifapp} we have for $z\in \RR^{N-1}\times [-R,R]$, 
\begin{align} \label{e.poly}
\left|k_m(z) - \frac{1}{T_m}b^\eta_0(T_mz^\prime)- b^\eta_1(T_m z^\prime) z_{N}\right| &= \frac{1}{T_m}\left| h_\nu(T_m z) - b^\eta_0(T_m z)- b^\eta_1(T_m z^\prime) T_m z_N\right|\notag\\
&\leqslant \frac{\eta }{T_m}(|T_m z_N|+R) =\eta |z_N|+\frac{\eta R}{T_m}.  
\end{align}
With the choice of  $z = (z^\prime, 0)\in\Sigma_\nu$ in \eqref{e.unifapp} and \eqref{e.poly}, respectively, we conclude that 
\begin{equation}\label{e.b0bound}
|b^{\eta}_{0}(z^{\prime})|\leq \eta R\quad\text{and}\quad \frac{1}{T_m}|b^\eta_0(T_m z^\prime)|\leq \frac{\eta R}{T_m}\quad\text{for all $z'\in \RR^{N-1}$.}
\end{equation}
 From \eqref{e.metriccom}, \eqref{e.unifapp}, and \eqref{e.b0bound}, we see that for $z_{N}\in (0, R]$, 
\begin{equation*}
b^\eta_1(z^\prime)=\frac{1}{z_{N}}b^\eta_1(z^\prime)z_{N}\leq \frac{1}{z_{N}}\left[h_{\nu}(z)-b^{\eta}_{0}(z^\prime)+\eta(|z_{N}|+R)\right]\leq \sqrt{\Theta}+\eta\left(1+\frac{2R}{z_{N}}\right),
\end{equation*}
and 
\begin{equation*}
b^\eta_1(z^\prime)=\frac{1}{z_{N}}b^\eta_1(z^\prime)z_{N}\geq  \frac{1}{z_{N}}\left[h_{\nu}(z)-b^{\eta}_{0}(z^\prime)-\eta(|z_{N}|+R)\right]\geq \sqrt{\theta}-\eta\left(1+\frac{2R}{z_{N}}\right). 
\end{equation*}
Taking $z_{N}=R$, we infer that 
\begin{equation} \label{e.b1bounds}
\sqrt{\theta} -3\eta \leqslant b^\eta_1(z^\prime) \leqslant \sqrt{\Theta} + 3\eta. 
\end{equation}
Since $b^\eta_1$ is Bohr uniformly almost periodic, it follows that the limit 
\begin{align} \label{e.b1mean}
\overline{b}^\eta_1 :=\mu(b^{\eta}_{1})= \lim_{T \to \infty} \frac{1}{T^{N-1}} \int_{T \square_\nu} b^\eta_1( z^\prime)\,dz^\prime = \lim_{T \to \infty} \int_{\square_\nu} b^\eta_1(Tz^\prime) \,dz^\prime
\end{align}
exists. From \eqref{e.b1bounds}, it follows that $\sqrt{\theta} - 3\eta \leqslant \overline{b}^\eta_1 \leqslant \sqrt{\Theta} + 3\eta. $ This implies that up to a subsequence (not relabeled),
\begin{align}
\label{e.meanconv}
\overline{b}_1^{\eta} \xrightarrow[\eta\to 0]{}c(\nu) 
\end{align}
for some $c(\nu) \in [\sqrt{\theta},\sqrt{\Theta}]$. 
Fix $\al>0$. As the functions $b^\eta_1(T_m \cdot)$ Bohr two-scale converge to $\overline{b}^\eta_1$ as $m \to \infty$(see Remark \ref{r.mean}), for any fixed $T > 1,$ using Definition \ref{d.2sc} in the domain $\Omega = T\square_\nu,$ this entails that for every test function $\psi : T\square_\nu \times \RR^N \to \RR,$ that is continuous in the first variable and Bohr almost periodic in the second variable, we have 
\begin{align*}
\lim_{m \to \infty} \int_{T\square_\nu} b^\eta_1(T_m z^\prime) \psi(z^\prime, T_m z^\prime)\,dz^\prime = \int_{T\square_\nu} \mu \big(\overline{b}^\eta_1  \psi(z^\prime, \cdot)\big) \,dz^\prime = \overline{b}^\eta_1\int_{T\square_\nu} \mu(\psi(z^\prime, \cdot)) \,dz^\prime . 
\end{align*} 
In what follows, we use this framework with the choice of test functions given by $\psi(z^\prime, w) := e^{-i \chi \cdot z^\prime},$ for $\chi\in \RR^{N-1}$, which are independent of $w.$ For such a choice, 
\begin{align}
\label{e.testfct}
\lim_{m \to \infty} \int_{T\square_\nu} b^\eta_1(T_m z^\prime) e^{-i \chi \cdot z^\prime} \,dz^\prime = \overline{b}^\eta_1 \int_{T\square_\nu} e^{-i \chi \cdot z^\prime} \,dz^\prime.
\end{align}


\textbf{Step D.} To prove \eqref{e.convdstf}, in view of Step A it remains to show that $k(z) = c(\nu) z \cdot \nu$ for all $z\in \RR^N$. We note that this is immediate for all $z\in \Sigma_{\nu}$, by definition of $k(z)$. We recall the function $z \mapsto \tilde{q}(z) = \frac{k(z)}{z_N}$ from Step A, and notice that being the locally uniform limit of Bohr almost periodic functions, $\tilde{q}(\cdot,z_N)$ is Bohr almost periodic. Fix $R > 0$ and also fix $|z_N|\leq R, z_N \neq 0.$ 
Our strategy is to show that for all $\chi \in \RR^{N-1},$ we have
\begin{align} \label{e.orthog}
\mu\left((\tilde{q}(\cdot,z_N) - c(\nu))e^{-i \chi \cdot (\cdot)} \right) = 0,
\end{align} 
where, as in \eqref{e.meanvalue}, $\mu$ denotes the mean value of the almost periodic argument. If we can verify \eqref{e.orthog}, then by \eqref{l.fsnought} we deduce that $\tilde{q}(z^\prime, z_N) \equiv c(\nu)$. As $z_N$ is arbitrary on $[-R,R]\setminus \left\{0\right\}$, we conclude that $k(z)=c(\nu)z\cdot \nu$ for all $z\in \RR^{N-1}\times [-R,R]$. 

Let $\al>0$, and for $T>1$ fixed let $K:=T\square_\nu\times \left\{z_{N}\right\}$ be a compact subset of $\left\{z\cdot \nu>0\right\}$, and let $M=M(\al, K)$ be as in \eqref{e.ucon}. Let $m_{0}\geq M$ be such that for all $m\geq m _{0}$, in view of \eqref{e.testfct}, 
\begin{equation}\label{e.avg2}
\left|\int_{T\square_\nu} b^{\eta}_{1}(T_{m}z)e^{-i \chi\cdot z'}\, dz'-\overline{b}^{\eta}_{1}\int_{T\square_\nu}e^{-i\chi\cdot z'}\, dz'\right|<\al.
\end{equation}
For $m\geq m_{0}$, by \eqref{e.ucon}, \eqref{e.poly}, \eqref{e.b0bound}, and \eqref{e.avg2}, we have

\begin{equation*} 
\begin{aligned}
&\left| \frac{1}{T^{N-1}}\int_{T\square_\nu}\left( \tilde{q}(z^\prime, z_N) - c(\nu)\right)e^{-i \chi \cdot z^\prime} \,dz^\prime\right| \\ & \leqslant 
\left| \frac{1}{T^{N-1}} \int_{T\square_\nu} \left( \frac{k(z^\prime,z_N)}{z_{N}} - \frac{k_m(z)}{z_N} \right) e^{-i \chi \cdot z^\prime} \,dz^\prime \right|\\& \quad + \left| \frac{1}{T^{N-1}} \int_{T\square_\nu} \left( \frac{k_m(z)}{z_N} - \frac{1}{T_m} \frac{b^\eta_0(T_m z^\prime)}{z_N} - b_1^\eta (T_m z^\prime)   \right)e^{-i \chi \cdot z^\prime}\,dz^\prime \right|\\
& \quad + \left|\frac{1}{T^{N-1}}\int_{T\square_\nu} \frac{1}{T_m} \frac{b^\eta_0(T_m z^\prime)}{z_N}e^{-i \chi \cdot z^\prime} \,dz^\prime \right| + \left|\frac{1}{T^{N-1}}\int_{T\square_\nu} \left( b_1^\eta(T_m z^\prime) - \overline{b}_1^\eta \right)e^{-i \chi \cdot z^\prime}\,dz^\prime  \right|\\
&\quad + |\overline{b}^\eta_1 - c(\nu)|\\
&\leqslant \alpha + \frac{\eta}{T_m} + \frac{2\eta R}{T_m |z_N|} +  \left|\frac{1}{T^{N-1}}\int_{T\square_\nu} \left( b_1^\eta(T_m z^\prime) - \overline{b}_1^\eta \right)e^{-i \chi \cdot z^\prime}\,dz^\prime  \right| + |\overline{b}^\eta_1 - c(\nu)|\\
&= \alpha + C\left( \frac{\eta}{T_m} + \frac{\eta R}{T_m |z_N|} \right) +  \frac{\alpha}{T^{N-1}} + |\overline{b}^\eta_1 - c(\nu)| .
\end{aligned}
\end{equation*}
We first send $m \to \infty,$ so that $T_m \to \infty.$ Then letting $T \to \infty,$ we obtain that 
\begin{align*}
\Big|\mu\left(\big(\tilde{q}(\cdot,z_N) - c(\nu)\big)e^{-i \chi \cdot (\cdot)} \right) \Big|\leqslant  \alpha + |\overline{b}^\eta_1 - c(\nu)|.
\end{align*} 
Sending $\eta, \alpha \to 0$ completes the proof of \eqref{e.orthog}, where we used \eqref{e.meanconv}. 

\textbf{Step E.} The foregoing argument shows that $k(z)=c(\nu)z\cdot \nu$ for all $z \in (\RR^{N-1}  \times [-R,R]).$ As $R$ is arbitrary, and $k$ from Step A is defined in $\RR^N$ (i.e., independently of any truncation $R$), we conclude that $ c(\nu)$ is independent of $R$ and 
\begin{align*}
k(z) = c(\nu) z\cdot \nu. 
\end{align*}
Thus, for all $K\subseteq \RR^{N}\setminus \Sigma_{\nu}$ compact, 
\begin{equation*}
\lim_{m\rightarrow \infty} \sup_{z\in K} \left|\frac{k_{m}(z)}{z_{N}}-c(\nu)\right|=0, 
\end{equation*} 
and this implies \eqref{e.convdstf}. 

\end{proof}
For notational convenience, we define $H: \RR^{N}\times \RR^{N}\rightarrow \RR$ by
\begin{equation*}
    H(p, y):=\frac{1}{\sqrt{a(y)}}|p|.
\end{equation*}
With this notation, \eqref{e.planmet2} yields that, in the viscosity sense,
\begin{equation*}
\begin{cases}
    H(\nabla k_{m},T_{m}x)=1&\text{in $\mathcal{H}_{\nu}$,}\\
  k_{m}(x)=0&\text{on $\Sigma_\nu$.}
    \end{cases}
\end{equation*}
We also point out that $H(\cdot, y)$ is uniformly Lipschitz continuous. Indeed, by \eqref{h.nondeg}, we have 
\begin{equation}\label{e.hlip}
|H(p,y)-H(q,y)|= \frac{1}{\sqrt{a(y)}}|p-q|\leq \frac{1}{\sqrt{\theta}}|p-q|.
\end{equation}

Throughout the rest of the paper, we take all equalities and inequalities of PDEs to be in the viscosity sense, and refer the reader to \cite{users, primer} for an overview of viscosity solutions.  

We next present a comparison principle which is specifically tailored for the proof of Theorem \ref{t.planarint}. A more general version of this result is stated in \cite[Lemma 3.3]{FS} without proof (although the proof essentially follows the same lines as \cite[Lemma 3.1]{ASlevel}.) For completeness, we provide a self-contained proof of the result we need here:
\begin{lemma}\label{l.compv}
Let $\eta>0$, and let $u,v\in C(\RR^{N})$ satisfy
\begin{equation}\label{e.hjorder}
\begin{cases}
    H(\nu+\nabla u, y)<H(\nu+\nabla v, y) -\frac{\eta}{\sqrt{\theta}}&\text{in $\mathcal{H}_{\nu}$},\\
    u(y)\leq v(y)&\text{on $\Sigma_{\nu}$,}
    \end{cases}
\end{equation}
 with 
\begin{equation}\label{e.liminfc}
    \liminf_{|y|\to \infty} \frac{v(y)-u(y)}{|y|}\geq -\eta.
\end{equation}
Then 
\begin{equation*}
    u(y)\leq v(y)\quad\text{in $\mathcal{H}_{\nu}$.}
\end{equation*}

\end{lemma}

\begin{proof}
Due to the strict inequality in \eqref{e.hjorder}, there exists $\ve>0$ so that 
\begin{equation}\label{e.hjorderep}
    H(\nu+\nabla u, y)\leq H(\nu+\nabla v, y) -\frac{\eta+\ve}{\sqrt{\theta}}\quad \text{in $\mathcal{H}_{\nu}$}.
\end{equation}
For each $R>1$, we define
\begin{equation*}
  \psi_{R}(y):=(R^2+|y|^{2})^{1/2}-R\quad\text{and}\quad v^{R}(y):=v(y)+(\eta+\ve) \psi_{R}(y).  
\end{equation*}
Notice that 
\begin{equation}\label{e.psiRfacts}
    \norm{\nabla \psi_{R}}_{L^\infty(\RR^{N})}\leq 1\quad\text{and}\quad \lim_{|y|\to \infty} |y|^{-1}\psi_{R}(y)=1.
\end{equation} 
Owing to \eqref{e.hlip}, \eqref{e.hjorderep}, and \eqref{e.psiRfacts} we obtain 
\begin{align*}
    H(\nu+\nabla v^{R}, y)&\geq H(\nu+\nabla v,y)-\frac{(\eta+\ve)}{\sqrt{\theta}}\norm{\nabla \psi_{R}}_{L^{\infty}(\RR^{N})}\\
    &=H(\nu+\nabla v,y)-\frac{(\eta+\ve)}{\sqrt{\theta}} \geq H(\nu+\nabla u,y)\quad\text{in $\mathcal{H}_{\nu}$.}
\end{align*}

Moreover, since $v^{R}\geq v$, we have 
\begin{equation}\label{e.bndorder}
    u(y)\leq v^{R}(y)\quad\text{on $\Sigma_{\nu}$}, 
\end{equation}
and by \eqref{e.liminfc} and superadditivity of $\liminf$, we find 
\begin{align*}
\liminf_{|y|\to \infty} \frac{v^{R}(y)-u(y)}{|y|}&=\liminf_{|y|\to \infty}\frac{v(y)+(\eta+\ve)\psi^{R}(y)-u(y)}{|y|}\\
&\geq\liminf_{|y|\to \infty}\frac{v(y)-u(y)}{|y|}+(\eta+\ve)\lim_{|y|\to\infty} \frac{ \psi_{R}(y)}{|y|}\\
&\geq -\eta+(\eta+\ve) = \ve>0. 
\end{align*}

This implies that for $|y|$ sufficiently large, the numerator must be nonnegative, and thus there is a large ball $B_{M}$ so that 
\begin{equation}\label{e.eventualorder}
    u(y)\leq v^{R}(y)\quad\text{on $\overline{\mathcal{H}_{\nu}}\setminus B_{M}$}. 
\end{equation}

We may now apply the comparison principle for the Dirichlet problem of stationary Hamilton-Jacobi equations on bounded domains \cite[Theorem 3.3]{users} to conclude that by \eqref{e.eventualorder}, the comparison principle, and \eqref{e.bndorder}, 
\begin{align*}
    \sup_{\mathcal{H}_{\nu}} \left(u(y)-v^{R}(y)\right)_{+}= \sup_{B_{M}\cap \mathcal{H}_{\nu}} \left(u(y)-v^{R}(y)\right)_{+}&\leq \max_{\partial B_{M}\cap \overline{\mathcal{H}_{\nu}}} \left(u(y)-v^{R}(y)\right)_{+}\\
    &\leq \sup_{\Sigma_{\nu}} \left(u(y)-v^{R}(y)\right)_{+}=0. 
\end{align*}

This yields that 
\begin{equation*}
    u(y)\leq v^{R}(y)\quad\text{in $\mathcal{H}_{\nu}$}. 
\end{equation*}
Finally, as $\psi_{R}(y)\rightarrow 0$ pointwise as $R\to \infty$, we have $v^{R}(y)\rightarrow v(y)$ pointwise as $R\to\infty$, independent of $\eta$ and $\ve$, and thus
\begin{equation*}
    u(y)\leq v(y)\quad\text{in $\mathcal{H}_{\nu}$.}
\end{equation*}

\end{proof}

In order to conclude the statement of Theorem \ref{t.planarint} along the whole sequence $T\to \infty$, we refer to a result of the famous (unpublished) work of Lions-Papanicolaou-Varadhan \cite{LPV}, concerning the existence of periodic correctors.

\begin{theorem}\cite[Theorem 1]{LPV}\label{t.lpv}
For each $p\in\RR^{N}$, there exists a unique number $\overline{H}(p)$ and $u\in\textup{Lip}(\mathbb{T}^{N})$ (the set of Lipschitz continuous and $\mathbb{T}^{N}$-periodic functions) so that $u$ solves
\begin{equation}\label{e.corrector}
    H(p+\nabla u, x)=\overline{H}(p)\quad\text{in $\RR^{N}$} 
\end{equation}
in the viscosity sense. 
\end{theorem}

We note that the function $u$ satisfying \eqref{e.corrector} is clearly not unique (for any $M\in \RR$, the function $u+M$ is also a solution to \eqref{e.corrector}), but emphasize that Theorem \ref{t.lpv} guarantees $\overline{H}(p)$ is unique. 

Equipped with Lemma \ref{l.hompmp}, Lemma \ref{l.compv}, and Theorem \ref{t.lpv}, we now present the proof of Theorem \ref{t.planarint}. 

\begin{proof}[Proof of Theorem \ref{t.planarint}]
We first argue that the value of $c(\nu)$ must be unique in $\mathcal{H}_{\nu}$. Let us suppose, for the purposes of contradiction, that $c(\nu)$ is not unique, and define 
\begin{align}\label{e.c1def}
    c_{1}(\nu)
    :&=\inf\Big\{c(\nu)\in [\sqrt{\theta}, \sqrt{\Theta}]: \text{$\exists$ a subsequence $\left\{T_{m}\right\}$ such that }\\
    &\quad \quad \quad  \text{$\lim_{T_{m}\to \infty} k_{m}(x)=c(\nu)x\cdot \nu$ loc. uniformly in $\mathcal{H}_{\nu}$}\Big\}\notag,
\end{align}
and 
\begin{align*}
    c_{2}(\nu)
    :&=\sup\Big\{c(\nu)\in [\sqrt{\theta}, \sqrt{\Theta}]: \text{$\exists$ a subsequence $\left\{T_{m}\right\}$ such that }\\
    &\quad \quad \quad  \text{$\lim_{T_{m}\to \infty} k_{m}(x)=c(\nu)x\cdot \nu$ loc. uniformly in $\mathcal{H}_{\nu}$}\Big\}. 
\end{align*}
By definition, $c_{1}(\nu)<c_{2}(\nu).$
Setting
\begin{equation*}
    w_{1}(y):=h_{\nu}(y)-c_{1}(\nu)(y\cdot \nu)\quad\text{and}\quad w_{2}(y):=h_{\nu}(y)-c_{2}(\nu)(y\cdot \nu), 
\end{equation*}
we have 
\begin{equation*}
    H(c_{1}(\nu)\nu+\nabla w_{1},  y)=1\quad\text{and}\quad H(c_{2}(\nu)\nu+\nabla w_{2}, y)=1\quad\text{in $\mathcal{H}_{\nu}$.}
    \end{equation*}
    By 1-homogeneity of $H(\cdot, y)$, this implies that for  $\tilde{w}_{1}(y):=c_{1}(\nu)^{-1}w_{1}(y)$ and $\tilde{w}_{2}(y):=c_{2}(\nu)^{-1}w_{2}(y)$, we have 
\begin{equation}\label{e.strictstack}
    H(\nu+\nabla \tilde{w}_{2}, y)=\frac{1}{c_{2}(\nu)}<\frac{1}{c_{1}(\nu)}=H(\nu+\nabla \tilde{w}_{1}, y)\quad\text{in $\mathcal{H}_{\nu}$}.
\end{equation}
We now claim for any $K\subseteq \mathcal{H}_{\nu}$ compact, 
\begin{equation}\label{e.strictsub}
    \limsup_{T\to\infty} \frac{1}{T}w_{2}(Tx)\leq  0\quad\text{and}\quad \liminf_{T\to \infty} \frac{1}{T}w_{1}(Tx)\geq 0\quad\text{for all $x\in K$.}
\end{equation}

Indeed, by Lemma \ref{l.hompmp}, we know that for any sequence $T\to \infty$, there exists a subsequence such $\left\{T_{m}\right\}$ and $\bar{c}(\nu)\in [\sqrt{\theta}, \sqrt{\Theta}]$ so that $T_{m}^{-1}h_{\nu}(T_{m}x)\rightarrow \bar{c}(\nu)x\cdot \nu$ as $T_{m}\to \infty$, and $\bar{c}(\nu)\in [c_{1}(\nu), c_{2}(\nu)]$. This, in particular, implies that for any convergent subsequence, 
\begin{align*}
T_{m}^{-1}w_{1}(T_{m}x)&=T_{m}^{-1}h_{\nu}(T_{m}x)-c_{1}(\nu)x\cdot \nu\\
&=T_{m}^{-1}h_{\nu}(T_{m}x)-\bar{c}(\nu)x\cdot \nu+[\bar{c}(\nu) -c_{1}(\nu)](x\cdot \nu)\\&\geq  T_{m}^{-1}h_{\nu}(T_{m}x)-\bar{c}(\nu)x\cdot \nu,
\end{align*}
where the right hand side tends to 0 as $m \to \infty$. Taking $\liminf$ of both sides, we see that every subsequential limit, hence the full sequence, satisfies the second assertion of \eqref{e.strictsub}. The other inequality in \eqref{e.strictsub} follows by an analogous argument. In particular, taking $x\in \mathbb{S}^{N-1}\cap \mathcal{H}_{\nu}$ and $y=Tx$, we have 
\begin{equation}\label{e.1strictsub}
    \limsup_{|y|\to\infty} \frac{\tilde{w}_{2}(y)}{|y|}\leq 0\quad\text{and}\quad \liminf_{|y|\to\infty}\frac{\tilde{w}_{1}(y)}{|y|}\geq 0. 
\end{equation}

By Theorem \ref{t.lpv}, let $u$ be the periodic corrector corresponding to $\overline{H}(\nu)$, so that 
\begin{equation}\label{e.corr}
    \begin{cases}
    H(\nu+\nabla u,y)=\overline{H}(\nu)&\text{in $\RR^{N}$},\\
    \lim_{|y|\to \infty} \frac{u(y)}{|y|}=0. 
    \end{cases}
\end{equation}

We consider two cases: 
\noindent \emph{Case 1:} $\overline{H}(\nu)\in \left[\frac{1}{c_{2}(\nu)}, \frac{1}{c_{1}(\nu)}\right]$. Without loss of generality, we will assume $\overline{H}(\nu)<\frac{1}{c_{1}(\nu)}$ (if not, then we can repeat the following argument using that $\overline{H}(\nu)>\frac{1}{c_{2}(\nu)}$). We claim that there exists $\eta>0$ so that the hypotheses of Lemma \ref{l.compv} are satisfied. 

Indeed, for the functions $u^{\eta}(y):=u(y)+\eta(y\cdot \nu)$ and $\tilde{w}_{1}(y)$, by \eqref{e.strictstack}, \eqref{e.hlip}, and the assumption of Case 1 that $\bar{H}(\nu)<\frac{1}{c_{1}(\nu)}$, there exists $\eta>0$ sufficiently small so that 
\begin{align*}
    H(\nu+\nabla u^{\eta}, y)\leq H(\nu+\nabla u,y)+\frac{\eta}{\sqrt{\theta}}=\overline{H}(\nu)+\frac{\eta}{\sqrt{\theta}}&< -\frac{\eta}{\sqrt{\theta}}+\frac{1}{c_{1}(\nu)}\\
    &=-\frac{\eta}{\sqrt{\theta}}+H(\nu+\nabla \tilde{w}_{1},y)\quad\text{in $\mathcal{H}_{\nu}$}. 
\end{align*}
We note that the function $u\in\text{Lip}(\mathbb{T}^{N})$ is bounded, and for any $M>0$, the function $u-M\in \text{Lip}(\mathbb{T}^{N})$ and also satisfies \eqref{e.corrector}. We may thus assume assume without loss of generality that $u\leq 0$. This implies 
\begin{equation*}
    u^{\eta}(y)=u(y)+\eta(y\cdot \nu)=u(y)\leq 0=\tilde{w}_{1}(y)\quad\text{on $\Sigma_{\nu}$.}
\end{equation*}
Furthermore, by \eqref{e.1strictsub} and \eqref{e.corr}, 
\begin{equation*}
    \liminf_{|y|\to \infty} \frac{\tilde{w}_{1}(y)-u^{\eta}(y)}{|y|}\geq-\eta. 
\end{equation*}

By Lemma \ref{l.compv}, this yields
\begin{equation}\label{e.contra}
    u(y)+\eta(y\cdot \nu)\leq \tilde{w}_{1}(y)\quad\text{in $\mathcal{H}_{\nu}$.}
\end{equation}

If the infimum in \eqref{e.c1def} is achieved, then there exists a subsquence $\left\{y_{m}\right\}=\left\{T_{m}x\right\}$ so that $\lim_{m\to\infty} |y_{m}|^{-1}\tilde{w}_{1}(y_{m})=0$. Dividing \eqref{e.contra} by $|y|$, using that $y\cdot\nu>0$, and evaluating this inequality along this particular sequence, we have 
\begin{equation*}
    \eta\leq 0,
\end{equation*}
which is a contradiction. 

If the infimum in \eqref{e.c1def} is not achieved, then we know that for every $\ve>0,$ there exists a subsequence $\left\{y_{m}\right\}=\left\{T_{m}x\right\}$ so that $T_{m}^{-1}h_{\nu}(T_{m}x)\to (c_{1}(\nu)-\ve)x\cdot \nu$. In particular, the function $\tilde{w}_{1,\ve}(x):=(c_{1}(\nu)-\ve)^{-1}h_{\nu}(x)-x\cdot \nu$ solves 
\begin{equation*}
    \frac{1}{c_{1}(\nu)-\ve}=H(\nu+\nabla \tilde{w}_{1,\ve}, y)>H(\nu+\nabla u,y) \quad\text{in $\mathcal{H}_{\nu}$.}
\end{equation*}
We now repeat the above argument with $u^{\eta,\ve}(y):=u(y)+\left(\eta+\frac{\ve}{c_{1}(\nu)-\ve}\right) y\cdot \nu.$ We may again choose $\eta>0$ sufficiently small so that 
\begin{equation*}
\begin{cases}
    H(\nu+\nabla u^{\eta,\ve},y)< -\frac{\eta}{\sqrt{\theta}}+H(\nu+\nabla \tilde{w}_{1,\ve},y)&\text{in $\mathcal{H}_{\nu},$}\\
    u^{\eta, \ve}(y)\leq \tilde{w}_{1}(y)&\text{on $\Sigma_{\nu}$,}
    \end{cases}
\end{equation*}
and we can check that 
\begin{equation*}
    \liminf_{|y|\to\infty} \frac{\tilde{w}_{1,\ve}(y)}{|y|}\geq \frac{\ve}{c_{1}(\nu)-\ve},
\end{equation*}
which implies
\begin{equation*}
    \liminf_{|y|\to \infty} \frac{\tilde{w}_{1, \ve}(y)-u^{\eta, \ve}(y)}{|y|}\geq -\eta. 
\end{equation*}
By another application of Lemma \ref{l.compv}, we have 
\begin{equation*}
    u(y)+\left(\eta+\frac{\ve}{c_{1}(\nu)-\ve}\right)y\cdot\nu\leq \tilde{w}_{1,\eta}(y)\quad\text{in $\mathcal{H}_{\nu}$}. 
\end{equation*}
Dividing by $|y|$ and taking this along the particular subsequence $\left\{y_{m}\right\}$, we have
\begin{equation*}
    \eta+\frac{\ve}{c_{1}(\nu)-\ve}\leq 0, 
\end{equation*}
which is another contradiction. 

\noindent \emph{Case 2.} $\overline{H}(\nu)\notin \left[\frac{1}{c_{2}(\nu)}, \frac{1}{c_{1}(\nu)}\right]$. Without loss of generality, let us assume $\overline{H}(\nu)>\frac{1}{c_{1}(\nu)}$ (otherwise, a symmetric argument to handle the other alternative). We consider the function $\hat{w}_{1, \la}(y):=\la \tilde{w}_{1}(y)+(1-\la)u(y)$. We note that by the convexity of $H(\cdot, y)$,
\begin{align*}
H(\nu+\nabla \hat{w}_{1,\la},y)\leq \la H(\nu+\nabla \tilde{w}_{1},y)+(1-\la)H(\nu+\nabla u,y)&=\frac{\la}{c_{1}(\nu)}+(1-\la)\overline{H}(\nu)\\
&< \lambda \overline{H}(\nu) + (1 - \lambda) \overline{H}(\nu)\\
&<\overline{H}(\nu)\quad\text{in $\mathcal{H}_{\nu}$.}
\end{align*}
Moreover, by definition of $\tilde{w}_{1}$ and \eqref{e.metriccom}, we have  
\begin{equation}\label{e.lastgrowth}
    \limsup_{|y|\to \infty}\frac{\hat{w}_{1, \la}(y)}{|y|}=\limsup_{|y|\to\infty} \frac{\la \tilde{w}_{1}(y)+(1-\la)u(y)}{|y|}\leq \la\left(\frac{\sqrt{\Theta}}{c_{1}(\nu)}+1\right)=:\la M.  
\end{equation}
We now proceed by the same arguments as above. Let $u_{\eta}(y):=u(y)-\eta(y\cdot \nu),$ choose $\la>0,$ and then $\eta=\eta(\la)>0$ sufficiently small, so that 
\begin{equation*}
\begin{cases}
    H(\nu+\nabla \hat{w}_{1,\la},y)< -\frac{\eta+\la M}{\sqrt{\theta}}+H(\nu+\nabla u_{\eta}, y)&\text{in $\mathcal{H}_{\nu},$}\\
    \hat{w}_{1,\la}(y)\leq u_{\eta}(y)&\text{on $\Sigma_{\nu}$.}
    \end{cases}
\end{equation*}

By \eqref{e.lastgrowth}, we have 
\begin{equation*}
    \liminf_{|y|\to \infty} \frac{u_{\eta}(y)-\hat{w}_{1,\la}(y)}{|y|}\geq -(\eta+\la) M. 
\end{equation*}

By Lemma \ref{l.compv}, this yields 
\begin{equation*}
    \hat{w}_{1,\la}(y)\leq u(y)-\eta(y\cdot\nu)\quad\text{in $\mathcal{H}_{\nu}$.}
\end{equation*}
Note that for any $\la>0$, if $\left\{y_{m}\right\}$ is a sequence such that $\lim_{|y_{m}|\to \infty} |y_{m}|^{-1}\tilde{w}_{1}(y_{m})=0$, then we have $\lim_{|y_{m}|\to \infty} |y_{m}|^{-1}\hat{w}_{1,\la}(y_{m})=0$. We then argue as in the last step of Case 1 to conclude that upon dividing by $|y|$ and sending $|y|\to\infty,$ we will have $0\leq -\eta.$ We note that we can make a similar argument as in Case 1 if the infimum in \eqref{e.c1def} is not achieved.

  Now since $c(\nu)$ is uniquely determined, we have that every subsequence $\left\{k_{m}\right\}=\left\{T_{m}^{-1}h_{\nu}(T_{m}\cdot)\right\}$ has a further subsequence which converges locally uniformly to $x\mapsto c(\nu)x\cdot \nu$, where $c(\nu)=\frac{1}{\overline{H}(\nu)}$ is uniquely defined. This implies that for any $K\subseteq \mathcal{H}_{\nu}$ compact, we have 
  \begin{equation*}
      \lim_{T\to\infty} \sup_{x\in K}\left|\frac{1}{T}h_{\nu}(Tx)-c(\nu)x\cdot \nu\right|=0. 
  \end{equation*}
  
  This, combined with Lemma \ref{l.hompmp}, guarantees that for any $K\subseteq \RR^{N}\setminus \Sigma_{\nu}$ compact, we have
    \begin{equation*}
      \lim_{T\to\infty} \sup_{x\in K}\left|\frac{1}{T}h_{\nu}(Tx)-c(\nu)x\cdot \nu\right|=0. 
  \end{equation*}
  
  We note that the functions $\left\{k_{m}\right\}$ are themselves uniformly bounded and uniformly equicontinuous on any compact set $K\subseteq \RR^{N}.$ In particular,  by the Arzel\`{a}-Ascoli Theorem, this implies that for any sequence $\left\{T_{m}\right\}$ tending to infinity, there exists a function $q(x)$ so that 
  \begin{equation*}
     \lim_{m\to\infty} \sup_{x\in K}|k_{m}(x)-q(x)|=0. 
  \end{equation*}
  In particular, by uniqueness of limits, we must have that $q(x)\equiv c(\nu)x\cdot \nu$, and this yields that for any $K\subseteq \RR^{N}$ compact, 
      \begin{equation*}
      \lim_{T\to\infty} \sup_{x\in K}\left|\frac{1}{T}h_{\nu}(Tx)-c(\nu)x\cdot \nu\right|=0. 
  \end{equation*}
We also have that 
      \begin{equation*}
      \lim_{T\to\infty} \sup_{x\in K}\left|\frac{1}{T}h_{-\nu}(Tx)-c(-\nu)(x\cdot -\nu)\right|=0. 
  \end{equation*}
  
 Since $h_{\nu}(x)=-h_{-\nu}(x)$ by \eqref{e.hodd}, taking $x\notin \Sigma_{\nu},$ we have 
\begin{align*}
|c(\nu)x\cdot\nu-c(-\nu)x\cdot \nu|&\leq  \left|-\frac{1}{T}h_{\nu}(Tx)+c(\nu)x\cdot \nu\right|+\left|\frac{1}{T}h_{\nu}(Tx)-c(-\nu)x\cdot \nu\right|\\
&=\left|\frac{1}{T}h_{\nu}(Tx)-c(\nu)x\cdot \nu\right|+\left|-\frac{1}{T}h_{-\nu}(Tx)-c(-\nu)x\cdot \nu\right|\\
&=\left|\frac{1}{T}h_{\nu}(Tx)-c(\nu)x\cdot \nu\right|+\left|-\frac{1}{T}h_{-\nu}(Tx)+c(-\nu)(x\cdot-\nu)\right|
\end{align*}

Taking a limit on the right as $T\to \infty$, we arrive at the conclusion that $c(\nu)=c(-\nu).$
\end{proof}

\begin{remark}
A posteriori, knowing that $c(\nu)=\frac{1}{\overline{H}(\nu)},$ we may use various known properties of effective Hamiltonian from \cite[Proposition 2]{LPV} to conclude properties of $c(\nu)$. For instance, we obtain that $c(\cdot)$ is Lipschitz continuous, and moreover that 
 \begin{equation*}
     \overline{H}(p):=\frac{1}{c(p/|p|)}|p|
 \end{equation*}
 is convex. One could pursue further analysis of alternative representation formulas for  $\overline{H}(\cdot)$ through its convex, homogeneous, and continuous nature (e.g. the analysis of supports and gauges, see \cite{rock}), but we do not carry this out here. 
 
 \end{remark}
 \begin{remark}\label{r.fg}
 We furthermore recall that in the works of \cite{Bangert, burago}, the authors identify the stable norm $\norm{x-y}_{*}$ for $x, y\in \mathbb{R}^{N}$, which represents a homogenized distance function between $x$ and $y$. If we think of $c(\nu)x\cdot\nu$ as the homogenized distance function to the plane $\Sigma_{\nu},$ then in comparison to the Euclidean setting, we expect that for $x\in \mathcal{H}_{\nu}$,
 \begin{equation*}
     c(\nu)x\cdot\nu=\inf_{y\in \Sigma_{\nu}} \norm{x-y}_{*}. 
 \end{equation*}
While we do not explore it here, we think it would be very interesting to connect Theorem \ref{t.planarint} with this related body of literature in geometry.
\end{remark}

Finally, we complete the paper with the proof of Theorem \ref{t.ap}. We first note an equivalent definition of Bohr almost periodicity, found in \cite{subin}, and used throughout the literature in almost periodic homogenization:
\begin{definition}\label{d.ap2}
A continuous, bounded function $g:\RR^{d}\rightarrow \RR$ is said to be \emph{Bohr almost periodic} if family of functions
\begin{equation*}
    \left\{g(\cdot+z): z\in \RR^{d}\right\}
\end{equation*}
is relatively compact in $\norm{\cdot}_{\infty}.$
\end{definition}

With this in hand, we are now ready to prove Theorem \ref{t.ap}:
\begin{proof}[Proof of Theorem \ref{t.ap}]
We note that if $a(\cdot)$ is (Bohr) almost periodic, then $a(\cdot)$ has a uniformly convergent Bochner-Fourier series (see \eqref{e.fourier}). This was the only fact we needed in the proof of Lemma \ref{l.almper}. By the equivalence of Definition \ref{d.ap2} and the characterization of Bohr almost periodic functions via their uniformly convergent Bochner-Fourier series (see \eqref{e.fourier}), Lemma \ref{l.almper} holds when $a(\cdot)$ is almost periodic. 

The proof of Lemma \ref{l.hompmp} only relies on the almost periodicity from Lemma \ref{l.almper}, in which case Lemma \ref{l.hompmp} also holds when $a(\cdot)$ is almost periodic. 

Finally, the proof of Theorem \ref{t.planarint} only uses the fact that for every $p\in \RR^{N}$, there exists a unique value of $\overline{H}(p)$ and a corrector $u$ which is bounded and solves 
\begin{equation}\label{e.loc}
    H(p+\nabla u,y)=\overline{H}(p)\quad\text{in $\RR^{N}$.}
\end{equation}
When $a(\cdot)$ is almost periodic, a result of Ishii \cite[Theorem 2]{Iap} yields the existence of a bounded uniformly continuous function $u$ satisfying \eqref{e.loc}. This implies that the conclusion of Theorem \ref{t.planarint} holds true when $a(\cdot)$ is almost periodic. 
\end{proof}

\appendix
\section{Modified Boundary Conditions via De Giorgi's Slicing Technique} \label{s.app}
Recalling the distance function $h_\nu$ introduced in Section \ref{ss.hnu}, we next argue that $\sigma$ has an alternative representation with boundary conditions in terms of the function $q\circ h_{\nu}$.
\begin{lemma} \label{l.newbc} 
Define $\overline{\sigma}: \mathbb{S}^{N-1} \to (0,\infty)$ by 
\begin{multline*}
\overline{\sigma}(\nu) :=  \lim_{T \to \infty} \frac{1}{T^{N-1}} \inf \Big\{\int_{TQ_\nu} \left[ a(y)W(u) + \frac{1}{2}|\nabla u|^2 \right] \,dy : \\ u \in H^1(TQ_\nu), u|_{\partial TQ_\nu} = q \circ h_{\nu} \Big\}.
\end{multline*}
Then $\overline{\sigma}(\nu) = \sigma(\nu)$ for all $\nu\in \mathbb{S}^{N-1}$. 
\end{lemma}
\begin{proof}


Fix $\nu\in \mathbb{S}^{N-1}$. For ease of notation, for $u\in H^{1}(A)$ and $A\subseteq \RR^{N}$ open, we introduce the localized functional
\begin{equation*}
\mathcal{G}_{\ve}(u, A):=\int_{A}\left[ \frac{1}{\ve}a\left(\frac{y}{\ve}\right)W(u) + \frac{\ve}{2}|\nabla u|^2 \right] \,dy.
\end{equation*} 

By the change of variables $y\mapsto Tx$ with $\ve=\frac{1}{T}$, we rewrite $\overline{\sigma}$ as 
\begin{equation}\label{e.sigbarep}
\overline{\sigma}(\nu)=\lim_{\ve \to 0} \inf \left\{\mathcal{G}_{\ve}(u, Q_\nu): u \in H^1\left(Q_\nu\right), u|_{\partial Q_\nu} = \left(q \circ h_{\nu}\right)\left(x/\ve\right) \right\},
\end{equation}
and, similarly, 
we rewrite \eqref{e.sigma} as
\begin{equation}\label{e.sigmaep}
\sigma(\nu) = \lim_{\ve\rightarrow 0} \inf \left\{\mathcal{G}_{\ve}(u, Q_\nu) : u \in H^1\left(Q_\nu\right), u|_{\partial Q_\nu} = \tilde{u}_{\rho, 1/\ve, \nu} \right\}. 
\end{equation}

To show that $\sigma(\nu) = \overline{\sigma}(\nu)$, we first prove that $\sigma(\nu)\geqslant  \overline{\sigma}(\nu)$. The inequality $\overline{\sigma}(\nu)\leqslant \sigma(\nu)$ can be carried out in an analogous manner. Throughout the rest of the proof, we let $Q:=Q_{\nu}$. Let $\eta_j \rightarrow0$ as $j \to \infty$ be fixed, and $\left\{v_j\right\}\subseteq H^1( Q)$ with $v_j(x) = \rho_{1/\eta_{j}} * u_{0,\nu}(x)$ on $\partial Q$ be such that  
\begin{equation}\label{e.minseq}
\lim_{j\rightarrow \infty}\mathcal{G}_{\eta_{j}}(v_{j}, Q)=\sigma(\nu). 
\end{equation}
We extend $v_{j}$ to $\tau Q$, for $\tau\in (1,2]$, by defining $v_{j}(x):=\rho_{1/\eta_{j}} * u_{0,\nu}(x)$ for $x\in \tau Q\setminus Q$.

\noindent Set $w_{j}(x):=\left(q\circ h_{\nu}\right)(x/\eta_{j})$ for $x\in \mathbb{R}^{N}$. By \eqref{e.tanhsech}, since $\tau\in (1,2]$, we have 
\begin{align*}
\begin{cases}
w_{j}\left(x\right)\geqslant 1-c_{1}e^{-c_{2}\frac{|x\cdot \nu|}{\eta_{j}}}&\text{if $x\in \tau Q_{\nu}\cap \left\{x\cdot \nu\geqslant  0\right\}$,}\\
w_{j}\left(x\right)\leqslant -1+c_{1}e^{-c_{2}\frac{|x\cdot \nu|}{\eta_{j}}}&\text{if $x\in \tau Q_{\nu}\cap \left\{x\cdot \nu<  0\right\}$.}
\end{cases}
\end{align*}
Therefore, as $|w_{j}|\leqslant 1$, 
\begin{align}\label{e.ujstep}
\lim_{j\rightarrow \infty}\norm{w_{j}-u_{0, \nu}}^{2}_{L^{2}(\tau Q)}&=\lim_{j\rightarrow \infty}\int_{\tau Q\cap \left\{x\cdot \nu>\sqrt{\eta_{j}}\right\}} |w_{j}-1|^{2}\, dx\\
&\quad +\lim_{j\rightarrow \infty}\int_{\tau Q\cap \left\{x\cdot \nu<-\sqrt{\eta_{j}}\right\}} |w_{j}+1|^{2}\, dx\notag\\
&\quad +\lim_{j\rightarrow \infty}\int_{\tau Q\cap \left\{-\sqrt{\eta_{j}}\leqslant x\cdot \nu\leqslant \sqrt{\eta_{j}}\right\}} 4\, dx\notag\\
&=0.
\end{align}
Clearly 
\begin{equation}\label{e.vjstep}
\lim_{j \to \infty} \|v_j - u_{0,\nu}\|_{L^2(\tau Q \setminus Q)} = 0. 
\end{equation} 
In particular, we have that 
\begin{equation}\label{e.ladef}
\la_{j}:=\norm{w_{j}-v_{j}}_{L^{2}(\tau Q\setminus Q)}\xrightarrow[j\rightarrow\infty]{}0.
\end{equation}


We will now construct a function $u_{j}\in H^{1}(\tau Q)$ such that $u_{j}\mid_{\partial(\tau Q)}=w_{j}$, and for some $C>0$,  
\begin{equation}\label{e.energytau}
\limsup_{j\rightarrow \infty}\int_{\tau Q}\left[\frac{1}{\eta_{j}}a\left(\frac{x}{\eta_{j}}\right)W(u_{j})+\frac{\eta_{j}}{2}|\nabla u_{j}|^{2}\right]\, dx\leqslant \sigma(\nu) + C(\tau-1).
\end{equation}
Our main approach will be to define a new function which smoothly interpolates between $v_{j}$ in $Q$ and $w_{j}$ on $\partial (\tau Q)$ in the region $\tau Q\setminus Q$. 

%

We note that by \eqref{e.tanhsech}, 
\begin{align}
|\nabla w_{j}(x)|^{2}&=2\left[W\left(\left(q\circ h_{\nu}\right)\left(\frac{x}{\eta_{j}}\right)\right)\right]\left|\nabla h_{\nu}\left(\frac{x}{\eta_{j}}\right)\right|^{2}\frac{1}{\eta_{j}^{2}}\notag\\
&\leqslant \frac{2\Theta}{\eta_{j}^{2}}\left[\left(1-\left(q\circ h_{\nu}\right)^{2}\left(\frac{x}{\eta_{j}}\right)\right)^{2}\right]\notag\\
&\leqslant  \frac{2\Theta}{\eta_{j}^{2}}\left[\left(1-\tanh^{2}\left(\sqrt{2}h_{\nu}\left(\frac{x}{\eta_{j}}\right)\right)\right)^{2}\right]\notag\\
&= \frac{2\Theta}{\eta_{j}^{2}}\left[\sech^4\left(\sqrt{2}h_{\nu}\left(\frac{x}{\eta_{j}}\right)\right)\right].\label{e.graddisplay}
\end{align}
Combining \eqref{e.tanhsech}, \eqref{e.metriccom}, and \eqref{e.graddisplay}, we find
\begin{equation}\label{e.wgrad}
|\nabla w_{j}|\leqslant \frac{C}{\eta_{j}}e^{-c\frac{|x\cdot \nu|}{\eta_{j}}}.
\end{equation}

We consider, for $k\in \mathbb{N}$ with $k>> \frac{1}{\tau-1}$, 
\begin{equation*}
L^{k}:=\left\{x\in \tau Q\setminus Q: \mathrm{dist}(x, \partial(\tau Q))<\frac{1}{k}\right\}. 
\end{equation*}
We divide $L^{k}$ into cubic shells (or layers), which we denote by $\left\{L^{k}_{i,j}\right\}_{i=1}^{M_{k,j}}$,  of thickness $\eta_{j}\la_{j}$, where $\la_{j}$ is defined by \eqref{e.ladef}. We note that $M_{k,j}:=\lceil \frac{1}{k\eta_{j}\la_{j}}\rceil$, and thus 
\begin{equation}\label{e.mkj}
M_{k,j}\eta_{j}\la_{j}\geqslant  1/k.
\end{equation} 

For every $k,j\in \NN$, we let $i_{0}\in \left\{1, \ldots, M_{k,j}\right\}$ be the smallest value such that 
\begin{equation} \label{e.least}
\int_{L_{i_{0},j}^{k}}b_{j}(x)\, dx\leqslant \frac{1}{M_{k,j}}\int_{L^{k}}b_{j}(x)\, dx,
\end{equation}
where 
\begin{equation*}
b_{j}(x):=\frac{1}{\eta_{j}}+\frac{|v_{j}(x)-w_{j}(x)|^{2}}{(\la_{j})^{2}\eta_{j}}+\eta_{j}(|\nabla v_{j}(x)|^{2}+|\nabla w_{j}(x)|^{2}).
\end{equation*}
We also consider cut-off functions $\vp_{k,j}\in C^{\infty}_{c}(\tau Q)$ with 
\begin{equation*}
0\leqslant \vp_{k,j}\leqslant 1, \quad \norm{\nabla \vp_{k,j}}_{L^{\infty}}\leqslant \frac{C}{\eta_{j}\la_{j}},
\end{equation*}
and 
\begin{equation*}
\vp_{k,j}=\begin{cases}
1&\text{for}\, \, x\in Q\bigcup \left(\bigcup_{i=1}^{i_{0}-1}L^{k}_{i,j}\right),\\
0&\text{for}\, \, x\in \left(\bigcup_{i=i_{0}+1}^{M_{k,j}}L^{k}_{i,j}\right). 
\end{cases}
\end{equation*}
We note that $\vp_{k,j}$ transitions precisely in the layer $L^{k}_{i_{0},j}$. 
We then set 
\begin{equation*}
\overline{u}_{k,j}:=\vp_{k,j}v_{j}+(1-\vp_{k,j})w_{j}, 
\end{equation*}
and we have by \eqref{e.ujstep} and \eqref{e.vjstep}, $\lim_{k\rightarrow \infty}\lim_{j\rightarrow\infty}\norm{\overline{u}_{k,j}-u_{0, \nu}}_{L^{2}(\tau Q\setminus Q)}=0$. 

We estimate 
\begin{align}\label{e.midenergy}
\mathcal{G}_{\eta_{j}}(\overline{u}_{k,j}, \tau Q)&=\mathcal{G}_{\eta_{j}}\left(v_{j}, \left(\bigcup_{i=1}^{i_{0}-1}L^{k}_{i,j}\right)\bigcup Q\right)+\mathcal{G}_{\eta_{j}}\left(\overline{u}_{k,j}, L^{k}_{i_{0},j}\right)+\mathcal{G}_{\eta_{j}}\left(w_{j}, \bigcup_{i=i_{0}+1}^{M_{i,j}}L^{k}_{i,j}\right)\notag\\
&=: A_{k,j}+B_{k,j}+C_{k,j}.
\end{align}
We have
\begin{equation*}
A_{k,j}\leqslant \mathcal{G}_{\eta_{j}}\left(v_{j}, Q\right)+\mathcal{G}_{\eta_{j}}\left(v_{j}, \tau Q\setminus Q\right),
\end{equation*}
and we see that, since $v_{j}=\rho_{1/\eta_{j}} * u_{0,\nu}(x)$ in $\tau Q\setminus Q$, 
\begin{align*}
\mathcal{G}_{\eta_{j}}(v_{j}, \tau Q\setminus Q)&=\int_{\tau Q \setminus Q\cap \left\{|x\cdot\nu|<\eta_{j}\right\}} \left[\frac{1}{\eta_{j}}a\left(\frac{x}{\eta_{j}}\right)W(v_{j})+\frac{\eta_{j}}{2}|\nabla v_{j}|^{2}\right]\, dx\\
&\leqslant \frac{C}{\eta_{j}}\left|\tau Q \setminus Q\cap \left\{|x\cdot\nu|<\eta_{j}\right\}\right|,
\end{align*}
where we used the facts that $v_{j}\in \left\{1,-1\right\}$ in $\tau Q\setminus Q\cap \left\{|x\cdot \nu|>\eta_{j}\right\}$, and that 
\begin{equation}\label{e.gradmol}
|\nabla v_{j}|=|\nabla (\rho_{1/\eta_{j}}\ast u_{0, \nu})|\leqslant C\eta_{j}^{-1}.
\end{equation}
Hence, 
\begin{equation*}
\mathcal{G}_{\eta_{j}}(v_{j}, \tau Q\setminus Q)\leqslant C(\tau-1)^{N-1}, 
\end{equation*}
which implies that
\begin{equation}\label{e.alimit}
\limsup_{k\rightarrow \infty}\limsup_{j\rightarrow \infty} A_{k,j}\leqslant \sigma(\nu)+C(\tau-1)^{N-1}.
\end{equation}

Next, since $\norm{v_{j}}_{L^{\infty}(\tau Q)}\leqslant 1$, and $\norm{w_{j}}_{L^{\infty}(\tau Q)}\leqslant C(\Theta)$, we have by \eqref{e.least}, 

\begin{align*}\label{e.b0}
B_{k,j}&\leqslant C\int_{L^{k}_{i_{0},j}}\left[\frac{1}{\eta_{j}}+\eta_{j}\left(|\nabla \vp_{k,j}|^{2}(v_{j}-w_{j})^{2}+|\nabla v_{j}|^{2}+|\nabla w_{j}|^{2}\right)\right]\, dx \notag\\
&\leqslant C\int_{L^{k}_{i_{0},j}}\left[\frac{1}{\eta_{j}}+\frac{|v_{j}-w_{j}|^{2}}{\eta_{j}\la_{j}^{2}}+\eta_{j}\left(|\nabla v_{j}|^{2}+|\nabla w_{j}|^{2}\right)\right]\, dx\notag\\
&\leqslant \frac{C}{M_{k,j}}\int_{L^{k}}\left[\frac{1}{\eta_{j}}+\frac{|v_{j}-w_{j}|^{2}}{\eta_{j}\la_{j}^{2}}+\eta_{j}\left(|\nabla v_{j}|^{2}+|\nabla w_{j}|^{2}\right)\right]\, dx.
\end{align*}
By \eqref{e.gradmol}, \eqref{e.wgrad}, and \eqref{e.mkj}, we obtain 
\begin{equation*}
B_{k,j}\leqslant \frac{C}{M_{k,j}}\left[\frac{1}{\eta_{j}}|L^{k}|+\frac{\la_{j}^{2}}{\eta_{j}\la_{j}^{2}}+\frac{C}{\eta_{j}}|L^{k}|\right]\leqslant \frac{C|L^{k}|}{\eta_{j}M_{k,j}}+\frac{C}{\eta_{j}M_{k,j}}\leqslant C\frac{\tau^{N-1}k\la_{j}}{k}+Ck\la_{j},
\end{equation*}
where we have used the fact that $|L^{k}|\leqslant C(N)\frac{\tau^{N-1}}{k}$. Hence, 
\begin{equation}\label{e.blimit}
\limsup_{k\rightarrow \infty}\limsup_{j\rightarrow \infty} B_{k,j}=0. 
\end{equation}

For the term $C_{k,j}$, we first remark that by \eqref{e.tanhsech}, and \eqref{e.metriccom}, 
\begin{align*}
W(w_{j})\leqslant (1-\tanh^{2}(2h_{\nu}(x/\eta_{j})))^{2}=\sech^{4}(2h_{\nu}(x/\eta_{j})))\leqslant Ce^{-c\frac{|x\cdot \nu|}{\eta_{j}}}.
\end{align*}

Combining this with \eqref{e.wgrad}, we have 
\begin{align*}
\int_{\tau Q\setminus Q} \left[\frac{1}{\eta_{j}}a\left(\frac{x}{\eta_{j}}\right)\right.&\left.W(w_j(x))+\frac{\eta_{j}}{2}|\nabla w_{j}(x)|^{2}\right]\, dx\\
&\leqslant \int_{\tau Q\setminus Q\cap \left\{|x\cdot \nu|\geqslant  2\eta_{j}\right\}} \left[\frac{1}{\eta_{j}}a\left(\frac{x}{\eta_{j}}\right)W(w_j(x))+\frac{\eta_{j}}{2}|\nabla w_{j}(x)|^{2}\right]\, dx\\
&\quad +\int_{\tau Q\setminus Q\cap \left\{|x\cdot \nu|\leqslant 2\eta_{j}\right\}} \left[\frac{1}{\eta_{j}}a\left(\frac{x}{\eta_{j}}\right)W(w_j(x))+\frac{\eta_{j}}{2}|\nabla w_{j}(x)|^{2}\right]\, dx\\
&\leqslant C\int_{\tau Q\setminus Q\cap \left\{|x\cdot \nu|\geqslant  2\eta_{j}\right\}} \left[\frac{1}{\eta_{j}}e^{-c\frac{|x\cdot \nu|}{\eta_{j}}}+\frac{1}{\eta_{j}}e^{-c\frac{|x\cdot\nu|}{\eta_{j}}}\right]\, dx\\
&\quad +C\int_{\tau Q\setminus Q\cap \left\{|x\cdot \nu|\leqslant 2\eta_{j}\right\}}\frac{1}{\eta_{j}}\, dx\\
&\leqslant C(\tau-1)^{N-1}+C\eta_{j}(\tau-1)^{N-1}\frac{1}{\eta_{j}}\\
&\leqslant C(\tau-1)^{N-1}. 
\end{align*}

This implies that 
\begin{equation}\label{e.climit}
\limsup_{k\rightarrow \infty}\limsup_{j\rightarrow \infty}C_{k,j}\leqslant C(\tau-1)^{N-1}. 
\end{equation}

Finally, by \eqref{e.alimit}, \eqref{e.blimit}, and \eqref{e.climit}, and using a diagonal argument, we may find an increasing sequence $\left\{k(j)\right\}$ such that 
\begin{equation*}
\limsup_{j\rightarrow \infty} \left[A_{k(j),j}+B_{k(j), j}+C_{k(j),j}\right]\leqslant \sigma(\nu)+C(\tau-1)^{N-1}.  
\end{equation*}
By \eqref{e.midenergy}, we let $u_{j}=\overline{u}_{k(j),j}$, and we arrive at \eqref{e.energytau}, with $u_{j}=w_{j}$ on $\partial (\tau Q)$.

To conclude, we now move from $\tau Q$ to $Q$, defining $\tilde{u}_{j}(x):=u_{j}(\tau x)$, for $x\in Q$, and changing variables to obtain
\begin{align*}
\int_{Q}\left[\frac{\tau}{\eta_{j}}a\left(\frac{\tau y}{\eta_{j}}\right)\right.&\left.W(\tilde{u}_{j}(y))+\frac{\eta_{j}}{2\tau}|\nabla \tilde{u}_{j}(y)|^{2}\right]\, dy\\
&=\frac{1}{\tau^{N}}\int_{\tau Q}\left[\frac{\tau}{\eta_{j}}a\left(\frac{x}{\eta_{j}}\right)W\left(\tilde{u}_{j}\left(\frac{x}{\tau}\right)\right)+\frac{\eta_{j}}{2\tau}\left|\nabla \tilde{u}_{j}\left(\frac{x}{\tau}\right)\right|^{2}\right]\, dx\\
&=\frac{1}{\tau^{N}}\int_{\tau Q}\left[\frac{\tau}{\eta_{j}}a\left(\frac{x}{\tau}\right)W\left(u_{j}(x)\right)+\frac{\eta_{j}\tau^{2}}{2\tau}\left|\nabla u_{j}(x)\right|^{2}\right]\, dx\\
&=\frac{\tau}{\tau^{N}}\int_{\tau Q}\left[\frac{1}{\eta_{j}}a\left(\frac{x}{\eta_{j}}\right)W\left(u_{j}(x)\right)+\frac{\eta_{j}}{2}\left|\nabla u_{j}(x)\right|^{2}\right]\, dx.
\end{align*}
By \eqref{e.energytau}, this implies 
\begin{equation*}
\limsup_{j\rightarrow \infty} \int_{Q}\left[\frac{\tau}{\eta_{j}}a\left(\frac{\tau y}{\eta_{j}}\right)W(\tilde{u}_{j}(y))+\frac{\eta_{j}}{2\tau}|\nabla \tilde{u}_{j}(y)|^{2}\right]\, dy\leqslant \frac{1}{\tau^{N-1}}\sigma(\nu)+\frac{C}{\tau^{N-1}}(\tau-1)^{N-1}. 
\end{equation*}
We note that $\ve_{j}:=\frac{\eta_{j}}{\tau}\rightarrow 0$ and for all $x\in \partial Q$, 
\begin{equation*}
\tilde{u}_{j}(x)=u_{j}(\tau x)=w_{j}(\tau x)=(q\circ h_{\nu})\left(\frac{\tau x}{\eta_{j}}\right)=(q\circ h_{\nu})\left(\frac{x}{\ve_{j}}\right).
\end{equation*}

Hence, $\tilde{u}_{j}$ is admissible for \eqref{e.sigbarep} (with sequence $\ve_{j}$), from which we conclude that
\begin{align*}
\overline{\sigma}(\nu)&\leqslant \limsup_{j\rightarrow \infty} \int_{Q}\left[\frac{\tau}{\eta_{j}}a\left(\frac{y}{\eta_{j}}\right)W(\tilde{u}_{j}(y))+\frac{\eta_{j}}{2\tau}|\nabla \tilde{u}_{j}(y)|^{2}\right]\, dy\\
&\leqslant \frac{1}{\tau^{N-1}}\sigma(\nu)+\frac{C}{\tau^{N-1}}(\tau-1)^{N-1}. 
\end{align*}
Letting $\tau\rightarrow 1^{+}$, we arrive at $\overline{\sigma}(\nu)\leqslant \sigma(\nu)$. As priorly mentioned, the opposite inequality follows from a symmetrical argument.

\end{proof}
\section*{Acknowledgements}

We thank William Feldman and Peter Morfe for their interest in our work, and very helpful discussions and correspondence.  

Rustum Choksi was supported by an NSERC Discovery Grant.  Irene Fonseca (I.F.) and Raghav Venkatraman (R.V.) acknowledge the Center for Nonlinear Analysis where part of this work was carried out. The research of I.F. was partially funded by the NSF grants DMS-1411646 and DMS-1906238. Jessica Lin was partially supported by an NSERC Discovery Grant, an FRQNT grant, and the Canada Research Chairs program. The research of R.V. was partially funded by NSF grant DMS-1411646. R.V. also acknowledges support from an AMS-Simons travel award.

\bibliographystyle{abbrv} 
 \bibliography{allencahn}
\end{document}